\documentclass[article,11pt]{amsart}
\usepackage[utf8]{inputenc}

\usepackage{graphicx, amsmath, amssymb, amscd, amsthm, euscript, psfrag, amsfonts,bm}

\usepackage{enumitem}
\usepackage{etoolbox}
\usepackage{amsmath}
\usepackage{enumerate}
\usepackage{amssymb}
\usepackage{amscd}
\usepackage{amsthm}
\usepackage{amsfonts}
\usepackage{graphicx}
\usepackage{tensor}
\usepackage{multirow}
\usepackage{multicol}
\usepackage{mathscinet}
\usepackage[T1]{fontenc}
\usepackage{mathrsfs}

\usepackage[colorlinks=true]{hyperref}
\hypersetup{urlcolor=black, linkcolor=black, citecolor=black}
\usepackage{cleveref}
\crefformat{section}{\S#2#1#3} % see manual of cleveref, section 8.2.1
\crefformat{subsection}{\S#2#1#3}
\crefformat{subsubsection}{\S#2#1#3}
\newenvironment{ack}{\subsection*{Acknowledgments}}{}

\newtheorem{theorem}{Theorem}[section]
\newtheorem{corollary}[theorem]{Corollary}

\newtheorem{lemma}[theorem]{Lemma}

\newtheorem{proposition}[theorem]{Proposition}
\newtheorem{conjecture}[theorem]{Conjecture}

\theoremstyle{definition}
\newtheorem{definition}[theorem]{Definition}
\newtheorem{example}[theorem]{Example}
\newtheorem{remark}[theorem]{Remark}

\newtheorem{question}[theorem]{Question}

\newcommand{\R}{\mathbb{R}} %shortcut for math letter
\newcommand{\Z}{\mathbb{Z}}
\newcommand{\Q}{\mathbb{Q}}

\newcommand{\ep}{\varepsilon}

\newcommand{\norm}[1]{\left\lVert#1\right\rVert}

\begin{document}
\title[Nondivergence of reductive group action]{Nondivergence of reductive group action on homogeneous spaces}
\author{Han Zhang}
\address{School of Mathematical Sciences, Soochow University, Suzhou 215006, China }
\email{hzhang.math@suda.edu.cn}

\author{Runlin Zhang}
\address{College of Mathematics and Statistics, Center of Mathematics, Chongqing University, 401331, Chongqing,  China. }
\email{runlinzhang@cqu.edu.cn}
\date{}

\maketitle
\begin{abstract}
Let $X=G/\Gamma$ be the quotient of a semisimple Lie group $G$ by its non-cocompact arithmetic lattice. Let $H$ be a reductive algebraic subgroup of $G$ acting on $X$. We give several equivalent algebraic conditions on $H$ for the existence of a fixed compact set in $X$ intersecting \textit{every} $H$-orbit. This generalizes previous results concerning certain special reductive group action on $X$ in this setting.

When $G$ is of real rank one, $\Gamma$ is a non-cocompact lattice of $G$ and $H<G$ is an algebraic group, we also obtain an algebraic condition on $H$ which is equivalent to the return of \textit{every} $H$-orbit to a single compact set in $X$. This complements our results in the case of arithmetic lattice. 

\end{abstract}

\tableofcontents

\section{Introduction}\label{Section: Introduction}
\subsection{Background}
Let $H<G$ be Lie groups and $\Gamma$ be a lattice of $G$. Consider the quotient space $X=G/\Gamma$, then the group multiplication induces a left group action of $H$ on $X$. If $X$ is noncompact, in many cases it is particularly important to understand whether or not an $H$-orbit intersects a given compact subset of $X$. For example, let $H=\{u_t:t\in \mathbb{R}\}$ be a one-parameter unipotent subgroup of $G$. In their fundamental work \cite{Dani_Margulis_1991_Asymptotic_behaviour_of_trajectories_of_unipotent_flows_on_homogeneous_spaces_MR1101994}, Dani and Margulis proved that given any $x\in X$, there exists a compact set $C\subset X$ depending on $x$ such that $\{t\in \mathbb{R}:u_t x \in C\}$ is unbounded. In particular, if this unipotent subgroup satisfies certain algebraic condition, then this compact set $C$ can be chosen uniformly for all $x\in X$. This result was crucially used in Ratner's proof of uniform distribution of trajectories of unipotent subgroups on homogeneous spaces \cite{Ratner_1991_raghunathanduke}. On the other hand, let $G$ be an algebraic group defined over $\mathbb{Q}$ and $\Gamma$ be an arithmetic subgroup of $G$. Let $H=T$ be a maximal $\mathbb{R}$-split torus of $G$ and $X=G/\Gamma$. Tomanov and Weiss \cite{Tomanov_Weiss_2003_Closed_orbits_for_actions_of_maximal_tori_on_homogeneous_spaces_MR1997950} proved that there exists a fixed compact set $C$ of $X$ intersecting every $T$-orbit. This result enables them to classify all $T$-closed orbits on $X$.
Motivatied by these results, in this article, we study the following nondivergence property of group action on homogeneous spaces:

\begin{definition}\label{Definition: uniform nondivergence}
The action of $H$ on $X$ is said to be \textit{uniformly non-divergent} if there exists a compact subset $C\subset X$ such that for \textit{every} $x\in X$, $H x\cap C\neq \emptyset$.
\end{definition}
\begin{remark}
  It is clear that the action of $H$ on $X$ is \textit{not} uniformly non-divergent if and only if there exists a sequence $\{x_n\}_{n\in \mathbb{N}}\subset X$ such that $Hx_n$ eventually leaves every compact subset of $X$.
\end{remark}

Now we take $X$ to be a homogeneous space of the form $G/\Gamma$, 
where $G$ is the connected component of the real points of a connected semisimple algebraic group $\boldsymbol{G}$ defined over $\mathbb{Q}$, and
 $\Gamma$ is an arithmetic lattice in $G$. Moreover, we take $H$ to be a closed subgroup of $G$.
The question that we wish to answer is the following :
\begin{question}
When is the action of $H$ on $X$ uniformly non-divergent?
\end{question}

Note that this property only depends on the $G$-conjugacy class of $H$.
There are several different approaches towards this question depending on the group $H$. 
\begin{itemize}
\item[1.] When $H$ is generated by unipotent flows, one may make use of the polynomial nature of $H$-orbits. See \cite{Dani_Margulis_1991_Asymptotic_behaviour_of_trajectories_of_unipotent_flows_on_homogeneous_spaces_MR1101994}.
\item[2.] When $H$ is a finitely generated Zariski dense subgroup of a semisimple subgroup without compact factors, one can study the random walk generated by a set of generators of this subgroup. See \cite{BenoistQuint2012, EskinMargulis2004}.
\item[3.] When $H$ is the real points of an $\R$-diagonalizable algebraic torus, one of the most successful approaches relies on some tools from algebraic topology initiated by McMullen \cite{McMullen_2005_Minkowski's_conjecture_well_rounded_lattices_and_topological_dimension_MR2138142}, and refined by Solan, Tamam \cite{Solan_2019_Stable_and_well_rounded_lattices_in_diagonal_orbits_MR4040835, Solan_Tamam_2022_On_topologically_big_divergent_trajectories}
 (cf. \cite{Levin_Shapira_Weiss_2016_Closed_orbits_for_the_diagonal_group_and_well_rounded_lattices_MR3605031, Shapira_Weiss_2016_Stable_lattices_and_the_diagonal_group_MR3519540} for similar ideas).
 See 
 \cite{Tomanov_2007_Values_Of_decomposable_forms_at_S_integral_points_and_orbits_of_tori_on_homogeneous_spaces_MR2322686,Tomanov_Weiss_2003_Closed_orbits_for_actions_of_maximal_tori_on_homogeneous_spaces_MR1997950,Weiss_2004_Divergent_trajectories_on_noncompact_parameter_spaces_MR2053601,Weiss_2006_Divergent_trajectories_and_Q_rank_MR2214461}
 for a different approach in this case.
\end{itemize}
These approaches indicate that certain \textit{algebraic obstruction} is the only obstruction to uniform nondivergence, which we explain now. In the following, we assume that $\boldsymbol{H}$ is an $\mathbb{R}$-algebraic subgroup of $\boldsymbol{G}$, and $H=\boldsymbol{H}(\mathbb{R})$.

When $H=\{\mathrm{id} \}$, the uniform nondivergence holds if and only if
$G/\Gamma$ is compact, which holds if and only if
$\boldsymbol{G}$ has no proper $\mathbb{Q}$-parabolic subgroups.
This is a hint that the general case may have to do with parabolic subgroups, and the formulation we find is linked with $\mathbb{Q}$-\textit{quasiparabolic subgroups} (see \cite[Definition 1.1]{BacThang2010}).

Assume that there exists an absolutely irreducible $\mathbb{Q}$-representation $\rho: \boldsymbol{G}\to \mathrm{GL}_n(\boldsymbol{V})$ and a nonzero rational vector $v\in \boldsymbol{V}(\mathbb{Q})$ such that 
\begin{itemize}
\item $\boldsymbol{H}$ fixes $v$;
\item $v$ is a highest weight vector.
\end{itemize}
In this case, we say that $\boldsymbol{H}$ is contained in a proper $\mathbb{Q}$-\textit{quasiparabolic subgroup} of $\boldsymbol{G}$. Now we claim that the $H$-action on $G/\Gamma$ is not uniformly non-divergent. This can be seen as follows.
Let $\delta : \boldsymbol{G}_m \to \boldsymbol{G}$ be a $\mathbb{Q}$-cocharacter that stabilizes the line spanned by $v$, and satisfies $\lim_{t\to 0} \rho(\delta(t)) \cdot v=0$.
Let $G_v$ be the stabilizer of $v$ in $G$.
By Mahler's criterion, 
\[
     H \delta_t \Gamma /\Gamma \subset  G_v \delta(t) \Gamma /\Gamma
     = \delta(t)  G_v \Gamma /\Gamma  
    \quad 
   \text{diverges as } t \text{ tends to }0,
\]
which implies the claim.

The approaches mentioned above are all able to show
that up to $G$-conjugacy, the converse holds in each case\footnote{ In the case of torus, \cite{Solan_2019_Stable_and_well_rounded_lattices_in_diagonal_orbits_MR4040835,Solan_Tamam_2022_On_topologically_big_divergent_trajectories} did not prove this exactly, but we will later explain how the result follows.} (namely, assume that $H$ is unipotent, finitely generated and Zariski dense in a connected semisimple group without compact factors, or an  $\mathbb{R}$-split algebraic torus).
In light of this, the following conjecture seems quite plausible:
\begin{conjecture}\label{Conjecture}
Let $\boldsymbol{G}$ be a semisimple $\mathbb{Q}$-algebraic group and let $G$ denote its real points.
Let $\Gamma$ be an arithmetic lattice in $G$ and $X:= G/\Gamma$.
Let $H$ be the real points of an $\R$-isotropic connected $\mathbb{R}$-algebraic subgroup $\boldsymbol{H}$ of $\boldsymbol{G}$.
The following two are equivalent:
\begin{itemize}
\item[1.] The action of $H$ is uniformly non-divergent on $X$;
\item[2.] For every $g\in G$, $g\boldsymbol{H}g^{-1}$ is not contained in a proper $\mathbb{Q}$-quasiparabolic subgroup of $\boldsymbol{G}$.
\end{itemize}
\end{conjecture}

Recall that an $\mathbb{R}$-algebraic group $\boldsymbol{H}$ is said to be $\mathbb{R}$-\textit{isotropic} if and only if in a Levi decomposition, its reductive part is an almost direct product of an $\mathbb{R}$-split torus and a semisimple group without compact factors.

To allow for a more general class of closed subgroups, one should allow in item 2 above a ``compact modification''. Moreover, Conjecture \ref{Conjecture} is wrong without the algebraicity assumption, see Question \ref{Question when H is nonalgebraic} below.

One may wish to compare this nondivergence criterion with the one obtained in
\cite[Theorem 1.7]{Zhang2021} 
(cf. \cite{Eskin_Mozes_Shah_1997_Nondivergence_of_translates_of_certain_algebraic_measures_MR1437473})
in a related but different context.

Unfortunately, combining approaches mentioned above in a naive way does not seem to yield a proof of the conjecture. Partial progress toward the above Conjecture \ref{Conjecture} has been made in our previous work \cite{Zhang_Zhang_2021_Nondivergence_on_homogeneous_spaces_and_rigid_totally_geodesics} and an application of this partial result has been found to obtain finiteness result of totally geodesic submanifolds with bounded volume \cite[Theorem 1.5]{Zhang_Zhang_2021_Nondivergence_on_homogeneous_spaces_and_rigid_totally_geodesics}. Nevertheless, in this paper we settle the conjecture in the affirmative when $\boldsymbol{H}$ is reductive (see Theorem \ref{Theorem: main theorem when A is algebraic}).
%Using \cite{BenoistQuint2012}, we can further replace the semisimple part of $H$ by a Zariski dense finitely generated subgroup.

We remark that Conjecture \ref{Conjecture} is also settled when $\boldsymbol{G}$ is $\mathbb{Q}$-split, since it reduces to the reductive case.
To see this, let us assume that the unipotent radical of $\boldsymbol{H}$ is nontrivial
and it suffices to explain why item $2$ implies item 1 in the conjecture above.
Let $\boldsymbol{H'}$ be the \text{observable hull} of $\boldsymbol{H}$ in $\boldsymbol{G}$, that is, the smallest subgroup of $\boldsymbol{G}$ containing $\boldsymbol{H}$ with the property that
\[
v \text{ is fixed by }\rho(\boldsymbol{H}) \implies
v \text{ is fixed by }\rho(\boldsymbol{H'})
\]
for every finite-dimensional representation $\rho$  of $\boldsymbol{G}$ and every vector $v$.
If $\boldsymbol{H'}$ is equal to $\boldsymbol{G}$, then it is called \textit{epimorphic} and by    \cite[Theorem 9]{Weiss_Finite-dimensional_representations_and_subgroup_actions_on_homogeneous_spaces}, the action of $H'$ is minimal and hence uniformly non-divergent. 
So let us assume that  $\boldsymbol{H'}$ is not equal to $\boldsymbol{G}$. Then by appealing to Sukhanov's theorem \cite{Sukhanov_The_description_of_observable_subgroups_of_linear_algebraic_groups} (see \cite[Theorem B]{BacThang2010}, the $(3)\iff (4)$ part) and note that our $\boldsymbol{G}$ is assumed to be $\mathbb{Q}$-split, $\boldsymbol{H}'$ is contained in a proper $\mathbb{Q}$-quasiparabolic subgroup of $\boldsymbol{G}$, which violates item $2$ in the conjecture.

In this article, we also study the case where $\Gamma$ is not an arithmetic lattice of $\boldsymbol{G}$ (see Theorem \ref{Theorem: nonarithmetic case}). By the Margulis arithmeticity theorem, it is necessary that the algebraic group $\boldsymbol{G}$ is of real rank 1. It turns out that in this case, a maximal $\mathbb{R}$-split torus of $\boldsymbol{G}$ play a crucial role in the uniform nondivergence property of subgroup actions.

One may also consider a stronger set-intersection property replacing a compact set $C$ by a ``deformation retract'' of the whole space. For a small sample of such research, see \cite{Levin_Shapira_Weiss_2016_Closed_orbits_for_the_diagonal_group_and_well_rounded_lattices_MR3605031,McMullen_2005_Minkowski's_conjecture_well_rounded_lattices_and_topological_dimension_MR2138142,Shapira_Weiss_2016_Stable_lattices_and_the_diagonal_group_MR3519540,Solan_2019_Stable_and_well_rounded_lattices_in_diagonal_orbits_MR4040835,Solan_Tamam_2022_On_topologically_big_divergent_trajectories}.

\subsection{Notations}
We will use the following conventions throughout:
\begin{itemize}
    \item Capitalized boldface letters $\boldsymbol{A},\boldsymbol{B},...$ and so on are often reserved for algebraic groups. The corresponding uppercase Roman letters $A,B,...$ denote their real points (if they are defined over $\mathbb{R}$). And lowercase Gothic letters $\mathfrak{a},\mathfrak{b},...$ are used for their (real) Lie algebras.
    \item For an algebraic group $\boldsymbol{A}$, let $\mathrm{X}(\boldsymbol{A})$ denote the character group of $\boldsymbol{A}$ consisting of algebraic group morphisms from $\boldsymbol{A}$ to $\mathbb{C}^{\times}$ over $\mathbb{C}$.
    Assume that $\boldsymbol{A}$ is contained in another algebraic group $\boldsymbol{B}$. For every $g\in \boldsymbol{B}$ and $\chi\in \mathrm{X}(g^{-1}\boldsymbol{A}g)$, we define a character $g(\chi)\in \mathrm{X}(\boldsymbol{A})$ by $g(\chi)(a):=\chi(g^{-1}a g)$ for every $a\in \boldsymbol{A}$. 
    \item For an algebraic group $\boldsymbol{A}$ (resp. a Lie group $A$), we let $\boldsymbol{A}^{\circ}$ (resp. $A^{\circ}$) denote its identity component with respect to the Zariski topology (resp. analytic topology). 
    \item For two groups living in some ambient group, $\mathrm{N}_{A}(B)$ (resp. $\mathrm{Z}_{A}(B)$) denotes the normalizer (resp. centralizer) of $B$ in $A$. If $\boldsymbol{A}$ and $\boldsymbol{B}$ are algebraic groups, we write $\boldsymbol{\mathrm{N}}_{\boldsymbol{A}}(\boldsymbol{B})$ or $\boldsymbol{\mathrm{Z}}_{\boldsymbol{A}}(\boldsymbol{B})$. The notation $\mathrm{Z}(A)$ (resp. $\boldsymbol{\mathrm{Z}}(\boldsymbol{A})$) denotes the center of a group $A$ (resp. an algebraic group $\boldsymbol{A}$). 
\end{itemize}
These notations should cause little confusion since for a connected algebraic group $\boldsymbol{B}$ over $\mathbb{R}$, the real points of $\boldsymbol{\mathrm{N}}_{\boldsymbol{A}}(\boldsymbol{B})$ or $\boldsymbol{\mathrm{Z}}_{\boldsymbol{A}}(\boldsymbol{B})$ coincide with  $\mathrm{N}_{A}(B)$ or $\mathrm{Z}_{A}(B)$ respectively. This is because the real points of a connected real algebraic group is Zariski dense. Same remarks apply to the center operation.

Let us now fix the playground.
\begin{itemize}
    \item Throughout the paper, we let $\boldsymbol{G}$ be a connected semisimple $\mathbb{Q}$-algebraic group, and $\Gamma\subset G$ be an arithmetic lattice. 
\end{itemize}

In addition to this, we also fix the following data associated with $\boldsymbol{G}$:
\begin{itemize}
    \item 
    Let $\boldsymbol{T}$ be a maximal $\mathbb{R}$-split torus of $\boldsymbol{G}$ containing a maximal $\mathbb{Q}$-split torus $\boldsymbol{S}$.
    \item 
     Fix an ordering of $\mathbb{Q}$-simple roots. 
     Let $r:=\mathrm{rank}_{\mathbb{Q}}(\boldsymbol{G})$. Let $\boldsymbol{P}_1,\cdots,\boldsymbol{P}_r$ be the standard maximal parabolic $\mathbb{Q}$-subgroups of $\boldsymbol{G}$, and $\chi_1,\cdots,\chi_r$ be the corresponding $\mathbb{Q}$-fundamental weights. 
     Each $\chi_i$ may be viewed as a character on $\boldsymbol{P}_i$ or $\boldsymbol{T}$, for $1\leq i\leq r$ (for details, see Section \ref{Section: Preliminaries}). 
     \item
     Let $\mathrm{W}(G)\cong \mathrm{N}_G(T)/\mathrm{Z}_G(T)$ be an $\mathbb{R}$-Weyl group of $\boldsymbol{G}$;
    
    \item We fix a Cartan involution $\tau:\boldsymbol{G}\to \boldsymbol{G}$ such that $\tau(a)=a^{-1}$ for any $a\in \boldsymbol{T}$.

\end{itemize}
    
 Later we will consider an algebraic group $\boldsymbol{M}$ and a maximal $\mathbb{R}$-split torus $\boldsymbol{D}$ in $\boldsymbol{\mathrm{Z}_G(M)}$. In this case, we let $\mathrm{W}(\mathrm{Z}_G(M))\cong \mathrm{N}_{\mathrm{Z}_G(M)}(D)/\mathrm{Z}_{\mathrm{Z}_G(M)}(D)$ be an $\mathbb{R}$-Weyl group of $\boldsymbol{\mathrm{Z}}_{\boldsymbol{G}}(\boldsymbol{M})$.

\subsection{Main results}
One of our main results, which confirms special cases of the Conjecture \ref{Conjecture}, is the following:

\begin{theorem}\label{Theorem: main theorem when A is algebraic}
Let $\boldsymbol{M}$ be a connected semisimple $\mathbb{R}$-algebraic subgroup of $\boldsymbol{G}$ without compact factors, and $\boldsymbol{A}$ be an $\mathbb{R}$-split torus in $\boldsymbol{\mathrm{Z}_G(M)}$.
Let $\boldsymbol{H}=\boldsymbol{A}\boldsymbol{M}$ and $\boldsymbol{D}$ be a maximal $\mathbb{R}$-split torus of $\boldsymbol{\mathrm{Z}_G(M)}$ containing $\boldsymbol{A}$.
Then 
the following statements are equivalent:
\begin{itemize}
    \item[(i)] The action of $H$ on $G/\Gamma$ is \textit{not} uniformly non-divergent;
    %\item[(ii)] For every $g\in G$ such that $g^{-1}\boldsymbol{D}g\subset \boldsymbol{T}$, the following holds: There exist nonempty $I\subset\{1,\cdots,r\}$, $w\in \mathrm{W}(G)$, and $w'\in \mathrm{W}(\mathrm{Z}_G(g^{-1}Mg))$ such that $g^{-1}\boldsymbol{M}g \subset \bigcap_{i\in I} w\boldsymbol{P_i}w^{-1}$,  and
   % $\{
   % w'w(\chi_i),\; i\in I
   % \}$ are linearly \textit{dependent} as (algebraic) characters\footnote{We say that a set of linear functionals are linearly dependent as algebraic characters if they are linearly dependent over $\mathbb{Z}$.} on $g^{-1}\boldsymbol{A}g$;
    
    %\item[(iii)] There exists $g\in G$ such that $g^{-1}\boldsymbol{D}g\subset \boldsymbol{T}$ and the following holds: There exist nonempty $I\subset\{1,\cdots,r\}$, $w\in \mathrm{W}(G)$, and $w'\in \mathrm{W}(\mathrm{Z}_G(g^{-1}Mg))$ such that $g^{-1}\boldsymbol{M}g \subset \bigcap_{i\in I} w\boldsymbol{P_i}w^{-1}$,  and
  %  $\{
  %  w'w(\chi_i),\; i\in I
 %   \}$ are linearly \textit{dependent} as (algebraic) characters on $g^{-1}\boldsymbol{A}g$;

    \item[(ii)] There exist $g\in G$ and a nonempty subset $I \subset \{ 1,\cdots,r\}$  such that $g^{-1}\boldsymbol{H}g \subset  \bigcap_{i\in I} \boldsymbol{P}_{i}$, and $\left\{
    \chi_i,\;i\in I
    \right\}$ are linearly \textit{dependent} as (\textit{algebraic}) characters\footnote{We say that a set of linear functionals are linearly dependent as algebraic characters (or characters for simplicity) if they are linearly dependent over $\mathbb{Z}$.} on $g^{-1}\boldsymbol{H}g$;
    
    \item[(iii)] There exist $g\in G$ and a connected reductive $\Q$-subgroup $\boldsymbol{L}$ of $\boldsymbol{G}$ containing $g^{-1}\boldsymbol{H}g$ such that $\boldsymbol{\mathrm{Z}_G{(L)}}/\boldsymbol{\mathrm{Z}(L)}$ is not $\Q$-anisotropic;
    
    \item[(iv)] There exist $g\in G$, a $\Q$-representation $\rho: \boldsymbol{G} \to \mathrm{GL}(\boldsymbol{V})$, and a vector $v\in \boldsymbol{V}(\Q)$ such that $0\in \overline{\rho(G)\cdot v}$ (i.e. $v$ is unstable) and $v$ is fixed by $g^{-1}H g$.
\end{itemize}
\end{theorem}

\medskip

{ Let us briefly mention some previous results in the setting of Theorem \ref{Theorem: main theorem when A is algebraic}. Tomanov and Weiss \cite{Tomanov_Weiss_2003_Closed_orbits_for_actions_of_maximal_tori_on_homogeneous_spaces_MR1997950} proved that if $H$ is any torus containing a maximal $\R$-split torus of $G$, then the action of $H$ on $G/\Gamma$ is uniformly non-divergent. In our earlier work \cite{Zhang_Zhang_2021_Nondivergence_on_homogeneous_spaces_and_rigid_totally_geodesics}, uniform non-divergence property was established for those reductive group $H$ with no compact factors satisfying that $\boldsymbol{\mathrm{Z}_G{(H)}}/\boldsymbol{\mathrm{Z}(H)}$ is $\R$-anisotropic. Both of these above mentioned results fall into the scope of Theorem \ref{Theorem: main theorem when A is algebraic}. }

\begin{remark}
We make some useful comments for Theorem \ref{Theorem: main theorem when A is algebraic}.
\begin{itemize}
   \item[(1)]  {Item (ii) in Theorem \ref{Theorem: main theorem when A is algebraic} can be regarded as a checkable criterion for item (i), while items (iii) and (iv) are algebraic characterizations of item (i).}

   \item[(2)] { As all the maximal $\R$-split tori in $\boldsymbol{G}$ are conjugated to each other, the following condition (ii') is equivalent to Theorem \ref{Theorem: main theorem when A is algebraic} item (ii). Hence it is worthwhile to note that (ii') can also be used as a criterion for uniform nondivergence property of $H$-action.
    }
    
\begin{itemize}
 \item[(ii')] { For every (equivalently, there exists) $g\in G$ such that $g^{-1}\boldsymbol{D}g\subset \boldsymbol{T}$, the following holds: There exist nonempty $I\subset\{1,\cdots,r\}$, $w\in \mathrm{W}(G)$, and $w'\in \mathrm{W}(\mathrm{Z}_G(g^{-1}Mg))$ such that $g^{-1}\boldsymbol{M}g \subset \bigcap_{i\in I} w\boldsymbol{P_i}w^{-1}$,  and
    $\{
    w'w(\chi_i),\; i\in I
    \}$ are linearly \textit{dependent} as (algebraic) characters on $g^{-1}\boldsymbol{A}g$. 
  }  
\end{itemize}

\end{itemize}

\end{remark}

Theorem \ref{Theorem: main theorem when A is algebraic} needs to assume $\boldsymbol{A}$ to be algebraic. We have the following more general Theorem \ref{Theorem: main theorem} dropping the algebraicity assumption on $\boldsymbol{A}$, whose item (2) implies that
 Theorem \ref{Theorem: main theorem when A is algebraic} does not hold without assuming $\boldsymbol{A}$ to be algebraic. This is because a set of linear functionals independent over $\mathbb{Z}$ is not necessarily independent over $\R$ (See Remark \ref{remark for main theorems} (2)). Indeed,  Theorem \ref{Theorem: main theorem when A is algebraic} will be deduced from Theorem \ref{Theorem: main theorem} in the next subsection.

\medskip

Recall that we have fixed a Cartan involution $\tau:\boldsymbol{G}\to \boldsymbol{G}$ such that $\tau(a)=a^{-1}$ for any $a\in \boldsymbol{T}$.
\begin{theorem}\label{Theorem: main theorem}
Let $\boldsymbol{M}$ be a semisimple $\mathbb{R}$-algebraic subgroup of $\boldsymbol{G}$ without compact factors  and $A$ be a Lie subgroup contained in $D$, where $\boldsymbol{D}$ is a maximal $\mathbb{R}$-split torus in $\boldsymbol{\mathrm{Z}_G(M)}$. Let $H=AM$.  Assume that $\boldsymbol{D}\subset \boldsymbol{T}$. Then the following statements are equivalent:
\begin{itemize}
\item[(1)] The action of $H$ on $G/\Gamma$ is \textit{not} uniformly non-divergent. 

%\item[(1')] The action of $H'$ on $G/\Gamma$ is \textit{not} uniformly non-divergent. 

\item[(2)] There exist $w\in \mathrm{W}(G)$, $w^{\prime}\in \mathrm{W}(\mathrm{Z}_G(M))$, and a nonempty subset $I\subset \{1,\cdots, r\}$ such that $w^{-1}\boldsymbol{M} w\subset  \bigcap_{i\in I} \boldsymbol{P}_{i} $, $w^{-1}\boldsymbol{M} w\subset  \bigcap_{i\in I} \tau(\boldsymbol{P}_i)$, and $\{w^{\prime}w(\chi_{i}):i\in I\}$ are linearly \textit{dependent} as linear functionals on $\mathrm{Lie}(A)$\footnote{i.e. they are linearly dependent over $\mathbb{R}$.}.
    
\item[(3)] For some $k\geq 1$, there exist linear $\mathbb{Q}$-representations $\rho_i:\boldsymbol{G}\to \mathrm{GL}(\boldsymbol{V}_i)$ with norms $\norm{\cdot}_i$ on $V_i:=\boldsymbol{V}_i(\mathbb{R})$, and nonzero vectors $v_i\in \boldsymbol{V}_i(\mathbb{Q})$ for $i=1,\cdots,k$, such that the following holds: For any $n\in \mathbb{N}$, there exists $g_n\in G$ such that for any $h\in H$, there exists $i\in \{1,\cdots,k\}$ with \[
\norm{\rho_i(h g_n)v_i}_i<\frac{1}{n}.
\]
\end{itemize}
\end{theorem}

\begin{remark}\label{remark for main theorems}
We have several comments for Theorem \ref{Theorem: main theorem}.
\begin{itemize}
\item[(1)] The assumption that $\boldsymbol{D}\subset \boldsymbol{T}$ makes the statement of Theorem \ref{Theorem: main theorem} (2) clean. It loses no generality because all maximal $\mathbb{R}$-split tori in $\boldsymbol{G}$ are conjugated to each other, and conjugation operation does not affect the uniform nondivergence property of ${H}$.

    \item[(2)] For condition (2) of Theorem \ref{Theorem: main theorem}, we note that when $A=\boldsymbol{A}(\mathbb{R})$ is algebraic, linear dependence of $\{w^{\prime}w(\chi_{i}):i\in I\}$ over $\R$ on $\mathrm{Lie}(A)$ is equivalent to linear dependence of $\{w^{\prime}w(\chi_{i}):i\in I\}$ over $\Z$ on $\mathrm{Lie}(A)$ (see Corollary \ref{Corollary: linear dependency in algebraic torus}). This equivalence does not hold when $A$ is not algebraic.

\item[(3)] Condition (3) of Theorem \ref{Theorem: main theorem} is an analog of  $(iv)$ in Theorem \ref{Theorem: main theorem when A is algebraic} in the situation where $H$ is nonalgebraic (equivalently, $A$ is nonalgebraic). Unlike the unique algebraic obstruction in $(iv)$ of Theorem \ref{Theorem: main theorem when A is algebraic}, one could find finitely many such obstructions when $H$ is nonalgebraic.

\end{itemize}
\end{remark}

We observe that uniform nondivergence property of $H$ as in Theorem \ref{Theorem: main theorem} is preserved if the semisimple part of $H$ is replaced by a Zariski dense finitely generated subgroup. 

\begin{corollary}\label{Cor: Zariski dense replacement}
{Under the assumptions of Theorem \ref{Theorem: main theorem}, let $\Lambda$ be a finitely generated Zariski dense subgroup of $M$ and $H'=A\Lambda$. Then the action of $H'$ on $G/\Gamma$ is uniformly non-divergent if and only if the action of $H$ on $G/\Gamma$ is uniformly non-divergent.
}
    
\end{corollary}

\begin{proof}
The direct implication is immediate since $H'\subset H$. For the converse, the proof follows from
\cite[Remark 5.2, Proposition 5.3]{BenoistQuint2012}. Let us complete the details below.

Let $B \subset G/\Gamma$ be a bounded set.
It suffices to show that there exists a possibly larger bounded set $B'$ of $G/\Gamma$ such that for every $x\in G/\Gamma$, 
\begin{equation*}
    M\cdot x \cap B \neq \emptyset 
    \implies \Lambda\cdot x \cap B' \neq \emptyset.
\end{equation*}
Without loss of generality, we assume that ${G}=\mathrm{SL}_n(\R)$ and $\Gamma=\mathrm{SL}_n(\Z)$ for some $n$.
For $\ep>0$, let $f_{\ep}: G/\Gamma \to[0,\infty]$ be a proper function as in \cite[Equation (5.1)]{BenoistQuint2012} ($H^{\text{nc}}$ there should be replaced by our $M$). By \cite[Remark 5.2]{BenoistQuint2012} and Mahler's criterion, we find $\ep_0>0$ such that 
\begin{equation*}
    M\cdot x \cap B \neq \emptyset \implies f_{\ep_0}(x) < \infty .
\end{equation*}
Then \cite[Proposition 5.3]{BenoistQuint2012} implies that there exists some $C_0>1$ such that for $x\in G/\Gamma$ satisfying $M\cdot x\cap B\neq \emptyset$ there exists $\gamma_x \in \Lambda$ such that
\begin{equation}\label{equa1}
    f_{\ep_0}(\gamma_x\cdot x) < C_0.
\end{equation}
As $f_{\ep_0}$ is a proper function,
\begin{equation*}
    B':= \left\{y\in G/\Gamma,\;  f_{\ep_0}(y)< C_0\right\}
\end{equation*}
is the desired bounded set.

\end{proof}

In light of Theorem \ref{Theorem: main theorem} and Conjecture \ref{Conjecture}, it is curious to ask the following:
\begin{question}\label{Question when H is nonalgebraic}
Let $\boldsymbol{G}$ be a semisimple $\mathbb{Q}$-algebraic group, and $G=\boldsymbol{G}(\mathbb{R})$. Let $\Gamma$ be an arithmetic lattice in $G$ and $X=G/\Gamma$. Let $H$ be a closed subgroup of $G$, not necessarily algebraic. Consider the following:
\begin{itemize}
\item[1.] The action of $H$ is not uniformly non-divergent on $X$;

\item[2.] Up to $G$-conjugacy class of $H$, Condition $(3)$ of Theorem \ref{Theorem: main theorem} holds.
\end{itemize}
Is item $1$ equivalent to item $2$ above?
\end{question}
By Proposition \ref{Proposition: Not uniformly non-divergent criterion}, it is clear that item $2$ implies item $1$. And Theorem \ref{Theorem: main theorem} gives a affirmative answer to the above question in the special case where $H=AM$, with $\boldsymbol{M}$ semisimple and $A\subset \mathrm{Z}_G(M)$.

We note the following immediate consequences of Theorem \ref{Theorem: main theorem}:

\begin{corollary}\label{Corollary 1}
Let $\boldsymbol{M}$, $\boldsymbol{D}$, $A$, and $H$ be as in Theorem \ref{Theorem: main theorem}. Assume that $\boldsymbol{D}\subset \boldsymbol{T}$. Suppose that the following holds: for any $w\in \mathrm{W}(G)$, any $w^{\prime}\in \mathrm{W}(\mathrm{Z}_G(M))$, and any nonempty subset $I\subset \{1,\cdots r\}$, if $\{w(\chi_{i}):i\in I\})$ are linearly independent as linear functionals on $\mathrm{Lie}(D)$, then $\{w^{\prime}w(\chi_{i}):i\in I\}$ are linearly independent as linear functionals on $\mathrm{Lie}(A)$. Then the action of $H$ on $G/\Gamma$ is uniformly non-divergent.
\end{corollary}

\begin{corollary}\cite[Theorem 1.2]{Zhang_Zhang_2021_Nondivergence_on_homogeneous_spaces_and_rigid_totally_geodesics}
Let $\boldsymbol{M}$, $\boldsymbol{D}$, $A$, and $H$ be as in Theorem \ref{Theorem: main theorem}. Assume that $A=D$, then the action of $H$ on $G/\Gamma$ is uniformly non-divergent.
\end{corollary}
\begin{proof}
If $\boldsymbol{A}=\boldsymbol{D}$, then condition $(1)$ in Theorem \ref{Theorem: main theorem} is automatically satisfied since $\mathrm{W}(\mathrm{Z}_G(M))$ preserves $D$ by conjugation. Therefore, the corollary follows.
\end{proof}

\begin{corollary}\label{coro: torus orbits}
Let $\boldsymbol{M}$, $\boldsymbol{D}$, $A$, and $H$ be as in Theorem \ref{Theorem: main theorem}. Assume that $\boldsymbol{D}\subset \boldsymbol{T}$ and $\boldsymbol{M}=\{\mathrm{id}\}$, so $H=A$ and $\boldsymbol{D}= \boldsymbol{T}$. Then the action of $A$ on $G/\Gamma$ is uniformly non-divergent if and only if for any $w\in \mathrm{W}(G)$, $w(\chi_1),\cdots,w(\chi_r)$ are linearly independent as linear functionals on $\mathrm{Lie}(A)$.
\end{corollary}
\begin{proof}
Note that when $\boldsymbol{M}=\{\mathrm{id}\}$, $\mathrm{W}(\mathrm{Z}_G(M))=\mathrm{W}(G)$. Therefore, the corollary follows by Theorem \ref{Theorem: main theorem}.
\end{proof}

To make our investigation complete, we also study the uniform nondivergence property of subgroup action on quotient $G/\Gamma$, where $\mathrm{rank}_{\mathbb{R}}(\boldsymbol{G})=1$. By the Margulis arithmeticity theorem, in this case $\Gamma$ could be a non-arithmetic lattice. The proof of Theorem \ref{Theorem: main theorem} crucially uses the arithmetic structure of $\Gamma$, which is not available when $\Gamma$ is non-arithmetic. Nevertheless, the reduction theory of Garland-Raghunathan \cite{Garland_Raghunathan_1970_Fundamental_domains_for_lattices_MR267041} allows us to establish the following theorem in the rank one case. The proof of this theorem  will be given in Section \ref{Section: nonarithmetic quotient}.

\begin{theorem}\label{Theorem: nonarithmetic case}
Let $\boldsymbol{G}$ be a connected semisimple algebraic group defined over $\mathbb{Q}$ with $\mathrm{rank}_{\mathbb{R}}(\boldsymbol{G})=1$, $\boldsymbol{H}$ be an $\mathbb{R}$-algebraic subgroup of $\boldsymbol{G}$, and $\Gamma$ be a lattice of $G$. Assume that $G/\Gamma$ is noncompact.
Then the action of $H$ on $G/\Gamma$ is uniformly non-divergent if and only if $\boldsymbol{H}$ contains a maximal $\mathbb{R}$-split torus of $\boldsymbol{G}$.
\end{theorem}

Theorem \ref{Theorem: nonarithmetic case} allows us to give an alternative proof of the following 'compact core' lemma \cite[Lemma 5.13]{Fisher_Lafont_Miller_Stover_2018_Finiteness_of_maximal_geodesic}, which plays an essential role in the analysis of dynamics in noncompact rank one locally symmetric spaces \cite{Bader_Fisher_Miller_Stover_2021_Arithmeticity_MR4250391,Fisher_Lafont_Miller_Stover_2018_Finiteness_of_maximal_geodesic}.

\begin{corollary}\label{Corollary: compact core}
 Given $1<m\leq n$. Let $\boldsymbol{G}=\boldsymbol{SO}(n,1)$ and $\boldsymbol{H}=\boldsymbol{SO}(m,1)\leq \boldsymbol{G}$. Let $\Gamma$ be a lattice in $G$. Then there exists a compact subset $C\subset G/\Gamma$ such that for any $x\in G/\Gamma$, $Hx\cap C\neq \emptyset$.
\end{corollary}

\begin{proof}
    By assumption, both $\boldsymbol{G}$ and $\boldsymbol{H}$ are of rank 1. In particular, $\boldsymbol{H}$ contains a maximal $\mathbb{R}$-split torus of $\boldsymbol{G}$. By Theorem \ref{Theorem: nonarithmetic case}, the action of $H$ on $G/\Gamma$ is uniformly non-divergent.
\end{proof}

\subsection{Proof of Theorem \ref{Theorem: main theorem when A is algebraic} assuming Theorem \ref{Theorem: main theorem}}

\begin{proof}[(i)$\implies$(ii)]
Conjugating by some $g\in G$, we assume that $\boldsymbol{D}\subset \boldsymbol{T}$. 

If the action of $H$ on $G/\Gamma$ is not uniformly non-divergent, then by Theorem \ref{Theorem: main theorem}, there exist $w\in \mathrm{W}(G)$, $w^{\prime}\in \mathrm{W}(\mathrm{Z}_G(M))$, and a nonempty subset $I\subset \{1,\cdots, r\}$ such that $w^{-1}\boldsymbol{M} w\subset  \bigcap_{i\in I} \boldsymbol{P}_{i} $, and  $\{w^{\prime}w(\chi_{i}):i\in I\}$ are linearly dependent as linear functionals on $\mathrm{Lie}(A)$. Since $\boldsymbol{A}$ is $\mathbb{R}$-algebraic and $A=\boldsymbol{A}(\mathbb{R})$, by Corollary \ref{Corollary: linear dependency in algebraic torus}, $\{w^{\prime}w(\chi_{i}):i\in I\}$ are linearly dependent as (algebraic) characters on $\boldsymbol{A}$.

%Conjugating by some $g\in G$, we assume that $\boldsymbol{D}\subset \boldsymbol{T}$ and $g=\mathrm{id}$ in $(iii)$.
Let
\begin{equation*}
    \boldsymbol{H}':= w^{-1}w'^{-1}\boldsymbol{H}w' w,\;
    \boldsymbol{A}':= w^{-1}w'^{-1}\boldsymbol{A}w' w,\;
    \boldsymbol{M}':= w^{-1}w'^{-1}\boldsymbol{M}w' w.
\end{equation*}
Also let $\boldsymbol{D}':=w^{-1}w'^{-1}\boldsymbol{D}w' w = w^{-1}\boldsymbol{D} w$, which is a maximal $\R$-split torus in $\boldsymbol{\mathrm{Z}}_{\boldsymbol{G}}(\boldsymbol{M}')$.
Then $\boldsymbol{M}' \subset \boldsymbol{P}_i$
 for every $i\in I$,
 $\boldsymbol{A}' \subset \boldsymbol{D}' \subset \boldsymbol{T}$,
 and that $\{\chi_i\}_{i\in I}$ are linearly dependent as characters on $\boldsymbol{A}'$. Since each $\chi_i$ is trivial restricted to the semisimple $\boldsymbol{M}'$, we have that $\{ \chi_i\}_{i\in I}$ are linearly dependent as characters on $\boldsymbol{H}'$. So we are done.
\end{proof}

\begin{proof}[(ii)$\implies$(iii)]
Replacing $g^{-1}\boldsymbol{H}g$ by $\boldsymbol{H}$, we assume that $g=\mathrm{id}$ in $(iii)$.
Write $\boldsymbol{P}_I= \boldsymbol{L}_I \ltimes \boldsymbol{U}_I$, where $\boldsymbol{L}_I $ is a Levi group defined over $\mathbb{Q}$ containing $\boldsymbol{T}$, and $\boldsymbol{U}_I$ is the unipotent radical of $\boldsymbol{P}_I$.
Then there exists $u\in U_I$ such that $u\boldsymbol{H}u^{-1}\subset \boldsymbol{L}_I$.
Replacing $\boldsymbol{H}$ by $u\boldsymbol{H}u^{-1}$, we may assume that $\boldsymbol{H}$ is contained in $\boldsymbol{L}_I$.

By assumption there are integers $\{l_i\}_{i\in I}$ such that $\prod_{i\in I} \chi_i^{l_i} = 1$ when restricted to $\boldsymbol{H}$. Thus 
\begin{equation*}
    \boldsymbol{H} \subset \boldsymbol{L}
    :=\left\{
    l\in  \boldsymbol{L}_I \,\middle\vert\,
    \prod_{i\in I} \chi_i^{l_i}(l)=1
    \right\}^{\circ}.
\end{equation*}
Note that $\boldsymbol{L}$ is a connected reductive $\Q$-subgroup. Hence it suffices to prove that $\boldsymbol{\mathrm{Z}}_{\boldsymbol{G}}(\boldsymbol{L})/\boldsymbol{\mathrm{Z}}(\boldsymbol{L})$ is not $\Q$-anisotropic, which holds if there is a $\Q$-cocharacter whose image centralizes $\boldsymbol{L}$ and yet is not contained in $\boldsymbol{L}$.

Indeed, $\{\chi_i\}_{i\in I}$ are linearly independent when restricted to $\boldsymbol{\mathrm{Z}}^{\text{spl}}(\boldsymbol{L}_I)$, the $\Q$-split part of the central torus of $\boldsymbol{L}_I$. Hence there exists a cocharacter $\delta: \mathbb{G}_m \to \boldsymbol{\mathrm{Z}}^{\text{spl}}(\boldsymbol{L}_I)$, which is automatically defined over $\Q$, such that 
\begin{equation*}
    \prod_{i\in I} \chi_i^{l_i}\circ \delta \neq 1.
\end{equation*}
Thus the image of $\delta$ centralizes $\boldsymbol{L}$ and is not contained in $\boldsymbol{L}$. So we are done.

Note that in the case where $\boldsymbol{A}=\boldsymbol{D}$, and so $\boldsymbol{H}=\boldsymbol{D}\boldsymbol{M}$, if there exist $w\in \mathrm{W}(G)$, and a nonempty subset $I\subset \{1,\cdots,r\}$ such that $w^{-1}\boldsymbol{H} w\subset \bigcup_{i\in I}\boldsymbol{P}_i$, then $\{w(\chi_i):i\in I\}$ are linearly independent as characters on $\boldsymbol{D}$. Otherwise, it would contradicts the the fact that $\boldsymbol{Z_G(H)}/\boldsymbol{Z(H)}$ is $\mathbb{R}$-anisotropic.

\end{proof}

\begin{proof}[(iii)$\implies$(iv)]
By assumption, we can find a $\Q$-cocharacter $\delta: \mathbb{G}_m \to \boldsymbol{G}$ whose image centralizes $\boldsymbol{L}$, yet is not contained in $\boldsymbol{L}$.
Let $U$ be the horospherical $\Q$-subgroup defined by this cocharacter and let $v$ be a $\mathbb{Q}-$vector in $\wedge^{\text{dim} U} \mathfrak{g}$ representing the Lie algebra of $U$. Then $v$ is a vector satisfying the conclusion.
\end{proof}

\begin{proof}[(iv)$\implies$(i)]
Replacing $g^{-1}\boldsymbol{H}g$ by $\boldsymbol{H}$, we assume that $g=\mathrm{id}$ in $(v)$. Since $0\in \overline{\rho(G)\cdot v}$, by \cite[Corollary 3.5, Theorem 4.2]{Kempf_1978_Instability_in_invariant_theory_MR506989}, we can find a $\mathbb{Q}$-cocharacter $\delta:\mathbb{G}_m\to \boldsymbol{G}$ such that $\rho(\delta(t))\cdot v\to 0$ as $t\to \infty$, and the image of $\delta$ centralizes $\boldsymbol{H}$. This implies that $(i)$ holds.
\end{proof}

\subsection{Examples}
\begin{example}
Let $\boldsymbol{G}$ be a semisimple algebraic group defined over $\mathbb{Q}$ satisfying $rank_{\mathbb{Q}}(\boldsymbol{G})=rank_{\mathbb{R}}(\boldsymbol{G})=r\geq 1$, and $\Gamma=\boldsymbol{G}(\mathbb{Z})$. Then $G/\Gamma$ is not compact (see e.g. \cite{Borel_HarishChandra_1962_Arithmetic_subgroups_of_algebraic_groups_MR147566}).

Let $\boldsymbol{T}$ be a maximal $\mathbb{R}$-split torus of $\boldsymbol{G}$, and $A\subset T$ be an $\mathbb{R}$-diagonalizable subgroup (not necessarily algebraic). If $A$ is a proper subgroup of $T$, then the set of all fundamental weights $\{\chi_1,\cdots,\chi_r\}$ are linearly dependent as linear functionals on $Lie(A)$, since $\dim A<r$. 

Therefore, by Theorem \ref{Theorem: main theorem}, we conclude that when $rank_{\mathbb{Q}}(\boldsymbol{G})=rank_{\mathbb{R}}(\boldsymbol{G})\geq 1$, the action of $A$ on $G/\Gamma$ is uniformly non-divergent if and only if $A=T$.
\end{example}

\begin{example}
Let $\mathbb{K}$ be a totally real field extension of $\mathbb{Q}$ with $[\mathbb{K}:\mathbb{Q}]=m$. Let $\boldsymbol{G}=\mathrm{Res}_{\mathbb{K}/\mathbb{Q}}(\boldsymbol{\mathrm{SL}}_n)$, where $\mathrm{Res}$ denotes Weil's restriction of scalar operator (see e.g. \cite[Chapter 2]{Platonov_Rapinchuk_1994_Algebraic_groups_and_number_theory_MR1278263}). Denote $G=\boldsymbol{G}(\mathbb{R})$, $\Gamma=\boldsymbol{G}(\mathbb{Z})$, and
\begin{align*}
    A=\{\mathrm{diag}(e^{t_1},\cdots,e^{t_n}):\sum_{i=1}^n t_i=0\}\subset \mathrm{SL}_n(\mathbb{R}).
\end{align*}
Let $\Delta:\mathrm{SL}_n(\mathbb{R})\to G$ be the diagonal embedding. Then the identity component of the real points of a maximal $\mathbb{Q}$-split torus $\boldsymbol{S}$ of $\boldsymbol{G}$ is $S=\Delta(A)$. Let $W$ be the Weyl group of $\mathrm{SL}_n(\mathbb{R})$ defined by $W\cong N_{\mathrm{SL}_n(\mathbb{R})}(A)/Z_{\mathrm{SL}_n(\mathbb{R})}(A)\cong S_n$, where $S_n$ is the usual symmetric group. Then the Weyl group $\mathrm{W}(G)=W\times\cdots \times W$, that is, $\mathrm{W}(G)$ is the product of $m$ copies of $W$.

The maximal $\Q$-torus $\boldsymbol{T}$ containing $\boldsymbol{S}$ can be decomposed uniquely into its $\Q$-anisotropic part and its $\Q$-split part $\boldsymbol{S}$.
Thus we have a projection from $\mathrm{Lie}(T)$ to $\mathrm{Lie}(S)$.
Note that $\mathbb{Q}$-fundamental weights $\{\chi_1,...,\chi_r\}$ on $\mathrm{Lie}(T)$ factor through its projection to $\mathrm{Lie}(S)$. Thus, if for some $w \in \mathrm{W}(G)$, the projection from $\mathrm{Ad}(w)\mathrm{Lie}(S)$ to $\mathrm{Lie}(S)$ is trivial, then 
$\{w(\chi_1),...,w(\chi_r)\}$ becomes trivial, hence linearly dependent, on $\mathrm{Lie}(S)$.
On the other hand, if the projection is surjective, then $\{w(\chi_1),...,w(\chi_r)\}$ is linearly independent on $\mathrm{Lie}(S)$.

Assume that $n=2$, and hence $\mathrm{rank}_{\mathbb{Q}}(\boldsymbol{G})=1$. It is easy to verify that when $m$ is even, there exists $w\in \mathrm{W}(G)$ such that $\mathrm{Ad}(w) \mathrm{Lie}(S)$ projects trivially to $\mathrm{Lie}(S)$, where the projection is with respect to the Killing form on $\mathrm{Lie}(T)$. Also, when $m$ is odd, one can verify that for any $w\in \mathrm{W}(G)$, $\mathrm{Ad}(w)\mathrm{Lie}(S)$ projects onto $\mathrm{Lie}(S)$.  By Corollary \ref{coro: torus orbits}, we conclude that the action of $S$ on $G/\Gamma$ is uniformly non-divergent if and only if $m$ is an odd number (cf. \cite{Tomanov_2013_Locally_divergent_orbits_on_Hilbert_modular_sapces_MR3038365,Tomanov_2021_Closures_of_locally_divergent_orbits_of_maximal_tori_and_values_of_homogeneous_forms_MR4308167} for a study on different problems in the similar setting).

When $n\geq 3$ and $m\geq 2$, one can always find $w\in \mathrm{W}(G)$ such that $\{w(\chi_1),...,w(\chi_r)\}$ are linearly dependent on $\mathrm{Lie}(S)$. 
We conclude that when $n\geq 3$, the $S$ action on $G/\Gamma$ is uniformly non-divergent iff $m=1$, or, $\mathbb{K}=\mathbb{Q}$.

For instance, when $n=3$ and $m=2$, the projection $\mathrm{Lie}(T)\to \mathrm{Lie}(S)$ can be written as
\[
\left(
\left[\begin{array}{ccc}
   t_1  & & \\
     & t_2& \\
     &&  -t_1-t_2
\end{array}\right],\,
\left[\begin{array}{ccc}
   s_1  & & \\
     & s_2& \\
     && -s_1-s_2
\end{array}\right]\right)
\]
mapped to the diagonal embedding of
\[
\left[\begin{array}{ccc}
   (t_1+s_1)/2  & & \\
     & (t_2+s_2)/2& \\
     && (-t_1-t_2-s_1-s_2)/2
\end{array}\right].
\]

 One can check that there does not exist $w\in \mathrm{W}(G)$ such that this projection  becomes trivial on $\mathrm{Ad}(w)\mathrm{Lie}(S)$.

However, if $w$ denotes the Weyl element which fixes the first coordinate, but swaps the first and the last entry in the second coordinate, then the two fundamental weights become, after applying $w$, linearly dependent on $\mathrm{Lie}(S)$. So we can still conclude that the action is not uniformly non-divergent by Corollary \ref{coro: torus orbits}.
\end{example}

\begin{example}
Let $\mathbb{K}$ be a totally real field extension of $\mathbb{Q}$ with $[\mathbb{K}:\mathbb{Q}]=2$. Let $\boldsymbol{G}=\mathrm{Res}_{\mathbb{K}/\mathbb{Q}}(\boldsymbol{\mathrm{SL}}_4)$. Let $\Gamma=\boldsymbol{G}(\mathbb{Z})$. Then $G=\mathrm{SL}_4(\mathbb{R})\times \mathrm{SL}_4(\mathbb{R})$. Consider the semisimple $\mathbb{R}$-algebraic group $\boldsymbol{M}\subset \boldsymbol{G}$ defined by 
\begin{align*}
    \boldsymbol{M}=\begin{bmatrix}  
    1 &  \\
     & \boldsymbol{\mathrm{SO}(2,1)}
    \end{bmatrix} \times
    \begin{bmatrix}
    1 & & & \\
    & 1 & &\\
    & & 1 & \\
    & & & 1
    \end{bmatrix},
\end{align*}
where $\boldsymbol{\mathrm{SO}(2,1)}$ is viewed as a subgroup of $\boldsymbol{\mathrm{SL}}_3$, which is embedded in the lower right block in $\boldsymbol{\mathrm{SL}}_4$. Note that a maximal $\mathbb{R}$-split torus $\boldsymbol{T}$ of $\boldsymbol{G}$ is the product of two maximal $\mathbb{R}$-split tori $\boldsymbol{T}'$ of $\boldsymbol{\mathrm{SL}}_4$, which are full diagonal tori. So $\mathrm{Lie}(T)=\{(x_1,x_2):x_i\in \mathrm{Lie}(T')\}$. 

A maximal $\mathbb{R}$-split torus $\boldsymbol{D}$ of $\boldsymbol{\mathrm{Z}_G(M)}$ has real points 
\begin{align*}
   D^{\circ}=\left\{\begin{bmatrix}
    e^{3t} & & &\\
    & e^{-t} & &\\
    & & e^{-t} &\\
    & & & e^{-t}
    \end{bmatrix}: t\in \mathbb{R}
    \right\}\times (T')^{\circ}.
\end{align*}
Let $\{\lambda_1,\lambda_2,\lambda_3\}$ be the set of all fundamental $\mathbb{Q}$-weights in $\boldsymbol{\mathrm{SL}}_4$ with respect to the full diagonal subgroup of $\boldsymbol{\mathrm{SL}}_4$, then $\{\chi_1,\chi_2,\chi_3\}$ is the set of all fundamental $\mathbb{Q}$-weights in $\boldsymbol{G}$, where for any $v=(x_1,x_2)\in\mathrm{Lie}(T)$,
\begin{align*}
    \chi_i(v)=\lambda_i(x_1)+\lambda_i(x_2).
\end{align*}
For $1\leq i\leq 3$, let $v_i\in\mathrm{Lie}(T)$ be such that $\chi_i(v)=(v_i|v)$ for any $v\in\mathrm{Lie}(T)$, where $(\cdot|\cdot)$ is the Killing form on $\mathrm{Lie}(T)$.
Denote by $\mathrm{W}(G)\cong \mathrm{N}_G(T)/\mathrm{Z}_G(T)$ and $\mathrm{W}(\mathrm{Z}_G(M))\cong \mathrm{N}_{\mathrm{Z}_G(M)}(D)/\mathrm{Z}_{\mathrm{Z}_G(M)}(D)$. For any $w\in \mathrm{W}(G)$, $w'\in \mathrm{W}(\mathrm{Z}_G(M))$ (implicitly we are fixing a choice of representatives of these Weyl groups with the understanding that the choice would not affect the discussion below), and nonempty $I\subset\{1,2,3\}$, define a linear subspace
\begin{align*}
    U(w,w',I):=\mathrm{Span}_{\mathbb{R}}\{w'w(v_i):i\in I\}\subset\mathrm{Lie}(T).
\end{align*}
Let 
\begin{align*}
    B:=\{U(w,w',I): w^{-1}\boldsymbol{M}w\subset \boldsymbol{P}_I, w'\in W(Z_G(M))\},
\end{align*}
then $B$ is a finite collection of linear subspaces of $\mathrm{Lie}(T)$.
For a linear subspace $V\subset \mathrm{Lie}(D)$ satisfying $\pi_{U(w,w',I)}(V)=U(w,w',I)$ for any $U(w,w',I)\in B$, let $A=\exp(V)\subset D$ (Since $\dim D=4$ and there are at most three linear functionals in $U(w,w',I)$, generic $3$-dimensional subspaces $V$ would meet this requirement). Then by Theorem \ref{Theorem: main theorem} and Corollary \ref{Corollary: If linear functionals are linearly independent, then ...}, the action of $H=AM$ is uniformly non-divergent on $G/\Gamma$. We also note that one can choose $V\subset \mathrm{Lie}(D)$ such that $A=\exp(V)$ is nonalgebraic.
\end{example}

\subsection{Overview of the proof strategy}
Here we indicate the strategy showing that statement (1) and (2) are equivalent in Theorem \ref{Theorem: main theorem}.

First we assume (2) does not hold. Let $H$ be as in Theorem \ref{Theorem: main theorem}. Denote $\pi:G \to G/\Gamma$. Note that $G/\Gamma=\bigcup_{\eta>0} X_{\eta}$ such that $\{X_{\eta}:\eta>0\}$ is a nested collection of compact sets, and for any compact set $C\subset G/\Gamma$, there exists $\eta_C>0$ with $C\subset X_{\eta_C}$. We want to prove that every $H$-orbit intersects a fixed compact set $X_{\eta_0}$ for some suitably chosen $\eta_0>0$. 

Note that $H=AM$, where $M$ is a semisimple group and $A$ is an $\mathbb{R}$-diagonalizable subgroup of $Z_G(M)$. Given $g\in G$, for any $h\in H$, as $h\pi(g)=am\pi(g)$, it is important to understand the structure of the set of elements $g\in G$ such that $M$ fails to bring $\pi(g)$ back to $K$.

For any $\eta>0$, consider the subset
\[G^M_{\eta}:=\{g\in G: M\pi(g)\cap X_{\eta}\neq \emptyset\}.\]
Then $G/\Gamma- \pi(G^M_{\eta})$ is the collection of elements in $G/\Gamma$ which can not be brought back by $M$ to the compact set $X_{\eta}$.

By the work of Daw-Gorodnik-Ullmo-Li \cite{Daw_Gorodnik_Ullmo_2021_The_space_of_homogeneous_probability_measures_on_is_compact}, $G/\Gamma-\pi(G_{\eta}^M)=\bigcup_{\boldsymbol{P}} \Sigma^M_{\eta,\boldsymbol{P}}$, where the union is taken over all $\mathbb{Q}$-parabolic subgroup $\boldsymbol{P}$ satisfying certain conditions determined by $M$, and $\Sigma^M_{\eta,\boldsymbol{P}}$ is certain generalized Siegel set (for a precise description, see the paragraph above Proposition \ref{Proposition: a description of cusps when taking consideration of M}). We wish to bring back any given element in $G/\Gamma-\pi(G_{\eta}^M)$ using subgroup $A$. 

It turns out that the structure of $\bigcup_{\boldsymbol{P}} \Sigma^M_{\eta,\boldsymbol{P}}$ leads to an open cover of $A$, thus an open cover on $\mathrm{Lie}(A)\cong \mathbb{R}^n$. However, this open cover is not good enough for us to apply a topological covering theorem of Euclidean spaces (Theorem \ref{Theorem: a covering theorem}), which essentially says that one can not cover an Euclidean space using a family of open sets with "low multiplicity". Thanks to a compactness criterion of Tomanov-Weiss (Proposition \ref{Proposition: compactness criterion}) and the fundamental result of Dani-Margulis on quantitative nondivergence of unipotent orbits (Theorem \ref{Theorem: quantitative nondivergence of unipotent orbits}), we are able to construct a good cover of $\mathrm{Lie}(A)$ (Lemma \ref{Lemma: a covering for torus}) to which the topological covering theorem applies. To finish the argument, we assume that statement (1) holds and prove by contradiction. For suitably chosen $\eta_0>0$, assume that there exists $x\in G/\Gamma-\pi(G_{\eta_0}^M)$ such that $Ax\cap X_{\eta_0}=\emptyset$. By some linear algebra argument combined with the covering theorem, this will leads to a contradiction.

To see that statement (2) implies (1), we first find a suitable Cartan involution (see Proposition \ref{Proposition: consequence of M contained in parabolic subgroup}) and again make use of some simple linear algebra argument (see Section \ref{Section: Some linear algebra lemmas}) to finish the proof. This part of proof is inspired by \cite[Example 1]{Tomanov_Weiss_2003_Closed_orbits_for_actions_of_maximal_tori_on_homogeneous_spaces_MR1997950}.

\subsection{Outline of the paper}
The rest of this article will be mainly devoted to proving Theorem \ref{Theorem: main theorem}. In Section \ref{Section: Preliminaries}, we will recall some basic notions and theorems from algebraic groups. In Section \ref{Section: Covering theorems}, we study properties of certain topological cover of the group ${G}$ and the $\mathbb{R}$-diagonalizable group $A$, which is constructed using the work in \cite{Daw_Gorodnik_Ullmo_2021_Convergence_of_measures_on_compactifications_of_locally_symmetric_spaces_MR4229603,Daw_Gorodnik_Ullmo_2021_The_space_of_homogeneous_probability_measures_on_is_compact}. Certain good properties of the cover is ensured by the fundamental result of Dani-Margulis on quantitative nondivergence of unipotent orbits \cite{Dani_Margulis_1991_Asymptotic_behaviour_of_trajectories_of_unipotent_flows_on_homogeneous_spaces_MR1101994}. In Section \ref{Section: Some linear algebra lemmas}, we prove some simple but useful linear algebra lemmas, which enable us to deal with the situation when $A$ is not necessarily algebraic. In Section \ref{Section: Proof of the main theorem}, we finish the proof of Theorem \ref{Theorem: main theorem}. There we make use of a topological covering theorem, which is initially introduced by McMullen in \cite{McMullen_2005_Minkowski's_conjecture_well_rounded_lattices_and_topological_dimension_MR2138142}, and later developed by Solan and Tamam \cite{Solan_2019_Stable_and_well_rounded_lattices_in_diagonal_orbits_MR4040835,Solan_Tamam_2022_On_topologically_big_divergent_trajectories}. In Section \ref{Section: nonarithmetic quotient}, we recall the reduction theory for real rank one quotient by Garland-Raghunathan, and give the proof of Theorem \ref{Theorem: nonarithmetic case}.

\section{Preliminaries}\label{Section: Preliminaries}

Recall that $\boldsymbol{G}$ is a linear algebraic semisimple group defined over $\mathbb{Q}$ and $\mathfrak{g}$ is the Lie algebra of its real points $G$. We also fix a norm $\norm{\cdot}$ on $\mathfrak{g}$. 
A Lie subalgebra of $\mathfrak{g}$ is said to be unipotent iff it corresponds to a (Zariski closed) unipotent subgroup of $\boldsymbol{G}$. 

Let $\Gamma$ be an arithmetic subgroup of $G$. Let $\mathrm{Ad}:G\to \mathrm{GL}(\mathfrak{g})$ be the adjoint representation of $G$ on its Lie algebra. By \cite{Borel_1966_Density_and_maximality_of_arithmetic_subgroups_MR205999}, we find a lattice $\mathfrak{g}_{\mathbb{Z}}$ of $\mathfrak{g}$ such that $\mathrm{Ad}(\Gamma)\mathfrak{g}_{\mathbb{Z}}=\mathfrak{g}_{\mathbb{Z}}$.
Let $\pi:G\to G/\Gamma$ be the natural projection map. For any $x=\pi(g)\in G/\Gamma$, denote $\mathfrak{g}_x=\mathfrak{g}_g=\mathrm{Ad}(g)\mathfrak{g}_{\mathbb{Z}}$. 
\subsection{Parabolic subgroups}
 Recall that $\boldsymbol{T}$ is a maximal $\mathbb{R}$-split torus of $\boldsymbol{G}$ containing a maximal $\mathbb{Q}$-split torus $\boldsymbol{S}$ and $r=\dim \boldsymbol{S}$ is the $\mathbb{Q}$-rank of $\boldsymbol{G}$. 
By \cite[21.8]{Borel_1991_Linear_algebraic_groups_MR1102012}, we can choose compatible orderings of $\mathbb{R}$-root system $\Phi_{\mathbb{R}}$ and $\mathbb{Q}$-root system $\Phi_{\mathbb{Q}}$. According to these orderings, we fix a $\mathbb{Q}$-minimal parabolic subgroup $\boldsymbol{P}_0$ containing $\boldsymbol{T}$. Let $\Delta_{\mathbb{Q}}$ be the set of all simple $\mathbb{Q}$-roots of $\boldsymbol{G}$. By \cite{Borel_1969_Introduction_aux_groupes_arithmetiques_MR0244260}, for some $t \in \R$, 
\begin{align}\label{align: Sigel set}
    G=K\cdot S_t\cdot C\cdot F^{-1}\cdot \Gamma,
\end{align}
where $K$ is a maximal compact subgroup, $C$ is a compact subset of $G$, $F\subset G(\mathbb{Q})$ is a finite subset, and
\begin{align*}
    S_t :=\{s\in S^{\circ}: \alpha(s)\leq t, \forall \alpha\in \Delta_{\mathbb{Q}}\}.
\end{align*}
For any subset $I\subset \Delta_{\mathbb{Q}}=:\{\alpha_1,\cdots,\alpha_r\}$, consider the standard parabolic $\mathbb{Q}$-subgroup $\boldsymbol{P}_I=\boldsymbol{\mathrm{Z}_G({S}_I)}\cdot \boldsymbol{N}$, where 
\begin{align*}
    \boldsymbol{S}_I :=(\bigcap_{\alpha\in \Delta_{\mathbb{Q}}\setminus I} \mathrm{Ker} ({\alpha}) )^{\circ},
\end{align*}
and $\boldsymbol{N}$ is a maximal unipotent subgroup contained in $\boldsymbol{P}_0$. In particular, $\boldsymbol{P}_0=\boldsymbol{P}_{\Delta_{\mathbb{Q}}}$. Define the finite collection
\begin{align}\label{Align: finite collection of parabolic subgroups}
    \mathcal{B}:=\{\lambda \boldsymbol{P}_I \lambda^{-1}: I\subset \Delta_{\mathbb{Q}},\lambda\in F\},
\end{align}
Then any parabolic $\mathbb{Q}$-subgroup $\boldsymbol{P}$ is conjugate to an element in $\mathcal{B}$ by some $\gamma\in \Gamma$ (see e.g. \cite{Dani_Margulis_1991_Asymptotic_behaviour_of_trajectories_of_unipotent_flows_on_homogeneous_spaces_MR1101994}).

For each $\alpha\in \Delta_{\mathbb{Q}}$, define a projection $\pi_{\alpha}:\Phi_{\mathbb{Q}}\to \mathbb{Z}$ by $\pi_{\alpha}(\chi)=n_{\alpha}$, where $\chi=\sum_{\beta\in \Delta_{\mathbb{Q}}}n_{\beta}\beta$.
Denote by $\mathfrak{g}_{\chi}$ the root space corresponding to $\chi\in \Phi_{\mathbb{Q}}$. Then by \cite[21.12]{Borel_1991_Linear_algebraic_groups_MR1102012},
\begin{align}\label{Align: Lie algebra decomposition of parabolic 1}
   \mathrm{Lie}(\mathrm{Rad}_{\mathrm{U}}(\boldsymbol{P}_I))=\bigoplus_{\exists \alpha\in I,\pi_{\alpha}(\chi)>0}\mathfrak{g}_{\chi},
\end{align}
and 
\begin{align}\label{Align: Lie algebra decomposition of parabolic 2}
   \mathrm{Lie}(\mathrm{Z}_G({S}_I))= \mathrm{Lie} (\mathrm{Z}_G({S}))\oplus \bigoplus_{\forall \alpha\in I,\pi_{\alpha}(\chi)=0}\mathfrak{g}_{\chi},
\end{align}
where $\mathrm{Rad}_{\mathrm{U}}(\boldsymbol{P}_I)$ is the unipotent radical of $\boldsymbol{P}_I$.

For each $i=1,\cdots,r$, let $\boldsymbol{P}_i=\boldsymbol{P}_{\{\alpha_i\}}$. Then $\boldsymbol{P}_1,\cdots,\boldsymbol{P}_r$ are standard maximal parabolic $\mathbb{Q}$-subgroups of $\boldsymbol{G}$ containing $\boldsymbol{P}_0$. Let $\mathfrak{u}_1,\cdots,\mathfrak{u}_r$ be the Lie algebra of the unipotent radical of $\boldsymbol{P}_1,\cdots,\boldsymbol{P}_r$, respectively. For each $j=1,\cdots,r$, let $\mathfrak{R}_j$ be the set of all the Lie algebras of unipotent radicals of parabolic $\mathbb{Q}$-subgroups $\boldsymbol{P}$ which are conjugated to $\boldsymbol{P}_j$. Let $\mathfrak{R}=\bigcup_{j=1}^r \mathfrak{R}_j$.

The following propositions are needed in the course of establishing our main theorems:
\begin{proposition}\label{Proposition: existence of Zassenhauss neighborhood}\cite[Proposition 3.3]{Tomanov_Weiss_2003_Closed_orbits_for_actions_of_maximal_tori_on_homogeneous_spaces_MR1997950}
There exists an open neighborhood $W_0$ of $0$ in $\mathfrak{g}$ such that for any $g\in G$, the Lie subalgebra generated by $\mathrm{Span}_{\mathbb{R}}(\mathfrak{g}_{g}\cap W_0)$ is unipotent.
\end{proposition}
The neighborhood $W_0$ in Proposition \ref{Proposition: existence of Zassenhauss neighborhood} is called a(n open) Zassenhaus neighborhood.

\begin{proposition}\label{Proposition: if the span is unipotent, then there exist standard parabolic}\cite[Proposition 5.3]{Tomanov_Weiss_2003_Closed_orbits_for_actions_of_maximal_tori_on_homogeneous_spaces_MR1997950}
Let $\mathfrak{v}_1,\cdots,\mathfrak{v}_k\in \mathfrak{R}$. Suppose that
the Lie subalgebra generated by $\mathrm{Span}_{\mathbb{R}}\{\mathfrak{v}_j:j=1,\cdots, k\}$ is unipotent in $\mathfrak{g}$, then there exists $g\in G$ and $\{i_1,\cdots,i_k\}\subset \{1,\cdots,r\}$ such that $\mathfrak{v}_j=\mathrm{Ad}(g)\mathfrak{u}_{i_j}, j=1,\cdots,k$.
\end{proposition}

\begin{proposition}\label{Proposition: compactness criterion}\cite[Proposition 3.5]{Tomanov_Weiss_2003_Closed_orbits_for_actions_of_maximal_tori_on_homogeneous_spaces_MR1997950}
For any subset $L$ of $G$, $\pi(L)\subset G/\Gamma$ is precompact if and only if there exists a neighborhood $W$ of $0$ in $\mathfrak{g}$ such that for every $g\in L$ and every $\mathfrak{u}\in \mathfrak{R}$, $\mathrm{Ad}(g)\mathfrak{u}\not\subset \mathrm{Span}_{\mathbb{R}}(\mathfrak{g}_{g}\cap W)$.
\end{proposition}

\begin{proposition}\label{Proposition: Not uniformly non-divergent criterion}
Let $k$ be a positive integer. For $i=1,\cdots,k$, let $\rho_i:\boldsymbol{G}\to \mathrm{GL}(\boldsymbol{V}_i)$ be linear representations of $\boldsymbol{G}$ defined over $\mathbb{Q}$, with norms $\norm{\cdot}_i$ on $V_i:=\boldsymbol{V}_i(\mathbb{R})$. Let $v_i\in \boldsymbol{V}_i(\mathbb{Q})$ be a nonzero vector for $i=1,\cdots,k$, and $H\subset G$ be a subgroup. Assume that for any $n\in \mathbb{N}$, there exists $g_n\in G$ such that for any $h\in H$, there exists $i\in \{1,\cdots,k\}$ with
\begin{align*}
    \norm{\rho_i(h g_n)v_i}_i<\frac{1}{n},
\end{align*}
then the action of $H$ on $G/\Gamma$ is \textit{not} uniformly non-divergent.
\end{proposition}
\begin{proof}
Assume the contrary, then there exists a compact subset $C\subset G/\Gamma$ such that for any $x\in G/\Gamma$, $Hx\cap C\neq \emptyset$. Since $C$ is compact in $G/\Gamma$, there exists a compact subset $\widetilde{C}\subset G$ such that $\pi(\widetilde{C})=C$, where $\pi:G\to G/\Gamma$ is the natural projection.

As $\rho_i$'s are $\mathbb{Q}$-representations of $\boldsymbol{G}$, and $v_i$'s are nonzero $\mathbb{Q}$-vectors, $\rho_i(\Gamma)v_i$ are discrete in $V_i$ for $i=1,\cdots,k$. In particular, since we consider only finitely many such representations, there exists $\epsilon_1>0$, such that 
\[
\min_{1\leq i\leq k}\inf_{\gamma\in \Gamma}\norm{\rho_i(\gamma)v_i}_i>\epsilon_1.
\]
By compactness of $\widetilde{C}$, there is $0<\epsilon_2<\epsilon_1$ such that 
\begin{align}\label{align: no small vecotrs if uniform non-divergent}
    \min_{i=1,\cdots,k}\inf_{g\in \widetilde{C}\Gamma}\norm{\rho_i(g)v_i}_i>\epsilon_2.
\end{align}
Choose $n\in \mathbb{N}$ such that $1/n<\epsilon_2$. By assumption of the proposition, we may find $g_n\in G$ such that for any $h\in H$, there is $1\leq i\leq k$ such that $\norm{\rho_i(hg_n)v_i}_i<1/n$. However, since the action of $H$ is uniformly non-divergent, for this $g_n$, there exists $h_n\in H$ such that $h_n g_n\in \widetilde{C}\cdot \Gamma$. By (\ref{align: no small vecotrs if uniform non-divergent}), we have for all $1\leq i\leq k$, $\norm{\rho_i(h_n g_n)v_i}_i>\epsilon_2$, which leads to a contradiction. 
\end{proof}

\subsection{Fundamental weights}\label{section: fundamental weights}
For each $j=1,\cdots,r$, let $\bigwedge^{d_j}\mathfrak{g}$ be the $d_j$-th wedge product of the Lie algebra of $G$, where $d_j$ is the dimension of $\mathfrak{u}_j$. We equip $\bigwedge^{d_j}\mathfrak{g}$ with the norm induced from the fixed norm $\norm{\cdot}$ on $\mathfrak{g}$, which we still denote by $\norm{\cdot}$ by abuse of notation. Let $p_j\in \bigwedge^{d_j}\mathfrak{g}$ be the wedge product of an integral basis of $\mathfrak{u}_j\cap \mathfrak{g}_{\mathbb{Z}}$. Similarly, for any $\mathfrak{u}\in \mathfrak{R}_j$, denote by $p_{\mathfrak{u}}\in \bigwedge^{d_j}\mathfrak{g}$ the wedge product of an integral basis of $\mathfrak{u}\cap \mathfrak{g}_{\Z}$. These vectors are well-defined up to sign.

Recall that $G$ acts on $\bigwedge^{d_j}\mathfrak{g}$ through $\rho_j:=\bigwedge^{d_j}\mathrm{Ad}$.
Let $V_j=\mathrm{Span}_{\mathbb{R}}\{\rho_j(g)p_j:g\in G\}$ be the irreducible  linear representation of $G$ with highest weight vector $p_j$ and highest weight $\chi_j$ with respect to $T$. The weights $\chi_1,\cdots,\chi_r$ are called \textit{fundamental weights}, which are linear functionals on the Lie algebra of $T$. For each $j$, denote by $\Phi_j$ the collection of all weights in the weight decomposition of $V_j$ with respect to $T$.

Since $\boldsymbol{G}$ is semisimple, the Killing form on $\mathfrak{g}$ restricts to a strictly positive definite and symmetric bilinear form $(\cdot|\cdot)$ on $\mathrm{Lie}(T)$ (see e.g. \cite{Helgason_1978_Differential_geometry_Lie_groups_and_symmetric_spaces_MR514561}). Therefore, for any linear functional $\lambda$ on $\mathrm{Lie}(T)$, there is a $v_{\lambda}\in\mathrm{Lie}(T)$ such that for any $v\in\mathrm{Lie}(T)$, $\lambda(v)=(v_{\lambda}|v)$.

By the diffeomorphism $\exp:\mathrm{Lie}(T)\to T^{\circ}$, for any linear functional $\lambda$ on $\mathrm{Lie}(T)$, by abuse of notation we also say that $\lambda$ is a linear functional on $T^{\circ}$  by setting
\begin{align}\label{align: character on T as character on Lie(T)}
    \lambda(a):=\lambda(v),
\end{align}
where $a=\exp(v)$, for any $a\in T^{\circ}$. From now on, when we say that some linear functionals $\lambda_1,\cdots,\lambda_k$ are linearly independent on $T^{\circ}$, we mean that they are linearly independent on $\mathrm{Lie}(T)$. This should not cause any confusion.

\subsection{Cartan involution}\label{section: Cartan involution}
Let $\mathfrak{g}=\mathfrak{k}\oplus\mathfrak{p}$ be a Cartan decomposition and $\theta:\mathfrak{g}\to \mathfrak{g}$ be the corresponding Cartan involution defined by $\theta(k+p)=k-p$ for $k\in \mathfrak{k},\; p\in \mathfrak{p}$. 
One can lift this Cartan involution $\theta$ to a global Cartan involution of $\boldsymbol{G}$, which we also denote by $\theta$.  Cartan involutions of $\boldsymbol{G}$ are unique up to conjugation by an element in $G$. As $\boldsymbol{T}$ is a maximal $\mathbb{R}$-split torus in $\boldsymbol{G}$, we can choose a Cartan involution $\tau$ such that $\tau(a)=a^{-1}$ for any $a\in T$. We refer the reader to \cite{Knapp_2002_Lie_groups_beyond_an_introduction_MR1920389} for more on Cartan involutions.

For a linear $\mathbb{R}$-representation $\rho:\boldsymbol{G}\to \mathrm{GL}(V)$, we can decompose $V=\bigoplus_{\chi\in \mathrm{X}(V)}V_{\chi}$
where $\boldsymbol{T}$ acts on each $V_{\chi}$ by some  character $\chi$ and $\mathrm{X}(V)$ is the collection of all characters with $V_{\chi}$ non-zero.
A Cartan involution $\rho(\tau)$ is induced on the image $\rho(\boldsymbol{G})$. Since $\rho(\boldsymbol{G})$ is semisimple, $\rho(\tau)$ can be extended to a Cartan involution on $\mathrm{GL}(V)$  (see e.g. \cite{Mostow_1955_Self_adjoint_groups_MR69830}), which we also denote by $\rho(\tau)$. 
Moreover, $\rho(\tau)$ induces an automorphism on $\mathrm{X}(V)$ by $\rho(\tau)(\chi)=-\chi$.
Also, for every $\chi\in \mathrm{X}(V)$, $v\in \rho(\tau)(V_{\chi})$ and $a\in T$, $\rho(a)v=-\chi(a)v$.

\section{Some Covering Theorems}\label{Section: Covering theorems}
The analysis of torus orbits on homogeneous spaces relies on certain covering theorems, which we will introduce below.
\subsection{A covering for $\mathbb{R}^n$}
\begin{definition}
The invariance dimension of a convex open set $U\subset \mathbb{R}^n$ is the dimension of its stabilizer in $\mathbb{R}^n$, that is,
\begin{align*}
    \mathrm{invdim}( U):=\dim \mathrm{Stab}_{\mathbb{R}^n}(U),
\end{align*}
where $\mathrm{Stab}_{\mathbb{R}^n}(U)=\{x\in \mathbb{R}^n:x+U=U\}$. By convention, we set $\mathrm{invdim} \emptyset=-\infty$.
\end{definition}

\begin{lemma}\label{Lemma: inequality of invariance dimension}\cite[Lemma 2.6]{Solan_2019_Stable_and_well_rounded_lattices_in_diagonal_orbits_MR4040835}
Let $U_1\subset U_2$ be open convex subsets of $\mathbb{R}^n$, then 
\begin{align*}
    \mathrm{invdim} (U_1)\leq \mathrm{invdim} (U_2).
\end{align*}
\end{lemma}

\begin{lemma}\label{Lemma: invariance dimension of convex set determined by k linearly independent functionals}
Given $k$ linearly independent linear functionals $\lambda_1,\cdots,\lambda_k$ on $\mathbb{R}^n$
and $k$ real numbers $a_1,\cdots,a_k\in \mathbb{R}$, we define
\begin{align*}
    U=\{x\in \mathbb{R}^n: \lambda_i(x)<a_i, \forall i=1,\cdots,k\}.
\end{align*}
Then $U$ is an open convex set with $\mathrm{invdim}(U)\leq n-k$.
\end{lemma}

\begin{proof}
The claim that $U$ is open convex follows by definition. Without loss of generality, we may assume that $0\in U$. 
Let $(\cdot|\cdot)$ be a strictly positive definite and symmetric bi-linear form on $\mathbb{R}^n$. 
For $i=1,\cdots,k$, let $v_i\in \mathbb{R}^n$ be such that $\lambda_i(x)=(v_i|x)$, for any $x\in \mathbb{R}^n.$
We claim that each $v_i$ is perpendicular to $\mathrm{Stab}_{\mathbb{R}^n}(U)$ with respect to $(\cdot|\cdot)$.

Since $\mathrm{Stab}_{\mathbb{R}^n}(U)$ is a vector space, for each $i$, we can write $v_i=v_i^1+v_i^2$, where $v_i^1\in \mathrm{Stab}_{\mathbb{R}^n}(U)$ and $v_i^2\in (\mathrm{Stab}_{\mathbb{R}^n}(U))^{\perp}$. As $0\in U$, for all $t\in \mathbb{R}$, $t v_i^1\in U$, and so
\begin{align*}
    (v_i|t v_i^1)=\lambda_i(t v_i^1)<a_i,\; \forall t\in \mathbb{R}.
\end{align*}
This can happen only if $(v_i|v_i^1)=0$, which implies $v_i^1=0$. This proves the claim and thus the lemma.
\end{proof}

\begin{theorem}\label{Theorem: a covering theorem}\cite[Theorem 1.4]{Solan_2019_Stable_and_well_rounded_lattices_in_diagonal_orbits_MR4040835}
Let $\mathfrak{U}$ be an open cover of $\mathbb{R}^n$. Assume that 
\begin{itemize}
    \item[(1)] The cover $\{\mathrm{conv}(U):U\in \mathfrak{U}\}$ is locally finite\footnote{A collection of subsets of $\mathbb{R}^n$ is locally finite if for any compact subset $C\subset \mathbb{R}^n$, there are only finitely many elements in the collection intersect $C$.}, where $\mathrm{conv}(U)$ denotes the convex hull of $U$;

\item[(2)] For every $k\leq n$ and $k$ different sets $U_1,\cdots, U_k\in \mathfrak{U}$,
\begin{align*}
    \mathrm{invdim}\text{ }\mathrm{conv}(U_1\cap U_2\cap \cdots\cap U_k)\leq n-k;
\end{align*}
\end{itemize}
Then there are $n+1$ elements in $\mathfrak{U}$ with nontrivial intersection.
\end{theorem}

\subsection{A covering for $G$}
In this subsection, we let $\boldsymbol{G}$, $\boldsymbol{M}$, $\boldsymbol{D}$ and $A$ be as in Theorem \ref{Theorem: main theorem}. Recall that $\pi:G\to G/\Gamma$ is the natural projection map. For $\eta>0$, define
\begin{align*}
    X_{\eta}:=\{\pi(g)\in G/\Gamma:\mathfrak{g}_g\cap W_{\eta}=\{0\}\},
\end{align*}
where $W_{\eta}$ is an open ball with radius $\eta$ in $\mathfrak{g}$ centered at $0$. By (generalized) Mahler's criterion, $X_{\eta}$ is a compact subset of $G/\Gamma$. Define \begin{align*}
    G^M_{\eta}:=\{g\in G: M\pi(g)\cap X_{\eta}\neq \emptyset\}.
\end{align*}
Fix a maximal compact subgroup ${K}$ of ${G}$.
For every parabolic $\mathbb{Q}$-subgroup $\boldsymbol{P}$, let $\boldsymbol{U_P}$ be the unipotent radical of $\boldsymbol{P}$. Then $\boldsymbol{P}/\boldsymbol{U_P}$ is a reductive $\mathbb{Q}$-group. Let $\boldsymbol{S_P}^{\prime}$ be the $\mathbb{Q}$-split part of the central torus of $\boldsymbol{P}/\boldsymbol{U_P}$. We fix lifts $\boldsymbol{S_P}$ and $\boldsymbol{A_P}$ of $\boldsymbol{S_P}^{\prime}$ to $\boldsymbol{P}$ such that $\boldsymbol{S_P}$ is  $\mathbb{Q}$-split and $\boldsymbol{A_P}$ is an $\mathbb{R}$-split torus invariant under the Cartan involution associated with ${K}$. Let $\Delta_P$ be the set of $\mathbb{Q}$-simple roots of $(\boldsymbol{S_P},\boldsymbol{P})$. As $\boldsymbol{A_P}$ is conjugate to $\boldsymbol{S_P}$ by a unique element in $\boldsymbol{U_P}$, we can also think of $\Delta_P$ as the set of simple roots for $(\boldsymbol{A_P},\boldsymbol{P})$.

Let $\boldsymbol{^0\!P}$ be the identity component of the subgroup of $\boldsymbol{P}$ defined by the common kernel of all $\mathbb{Q}$-characters of $\boldsymbol{P}$. 
Assume that there are $I\subset \{1,\cdots,r\}$ and $\lambda\in F$ 
($F$ is as in (\ref{align: Sigel set})) such that 
$\boldsymbol{P}=\boldsymbol{P}_I^{\lambda}:=\lambda \boldsymbol{P}_I \lambda^{-1}$.
By rational Langlands decomposition, for each $g\in G$, we can write
\[
g=k_g(I,\lambda) a_g(I,\lambda) p_g(I,\lambda)
\]
 with
 \[
 k_g(I,\lambda)\in K,\;a_g(I,\lambda)\in A_{P_I^{\lambda}}^{\circ},\;p_g(I,\lambda)\in ({^0\!P}_I^{\lambda})^{\circ}.
 \]
For any subset $I\subset \{1,\cdots,r\}$, $\lambda \in F$, bounded set $B\subset G$ and real numbers $\theta,\epsilon>0$, as in \cite{Zhang_Zhang_2021_Nondivergence_on_homogeneous_spaces_and_rigid_totally_geodesics}, we define
\begin{align*}
    \Sigma_{I,\lambda,B}^M(\theta)= \{  g\in G: g^{-1}\boldsymbol{M}g\subset \boldsymbol{P}_I^{\lambda}, \alpha(a_g(I,\lambda))<\theta,\forall \alpha\in \Delta_{P_I^{\lambda}}, \text{ and }\\
    \exists m\in M \text{ such that } p_g(I,\lambda) g^{-1}mg\in (B\cap  {^0\!P}_I^{\lambda})\cdot(\Gamma\cap {^0\!P}_I^{\lambda})  \},
\end{align*}
and 
\begin{align*}
    \Sigma_{I,\lambda,B}^M(\theta,\epsilon)=\{g\in \Sigma_{I,\lambda,B}^M(\theta):\exists \alpha\in \Delta_{P_I^{\lambda}},\alpha(a_g(I,\lambda))<\epsilon \}.
\end{align*}

\begin{proposition}\cite[Proposition 3.1]{Zhang_Zhang_2021_Nondivergence_on_homogeneous_spaces_and_rigid_totally_geodesics}\label{Proposition: a description of cusps when taking consideration of M}
There exist $0<\theta<1$ and a compact subset $B\subset G$ such that the following holds: there exist $0<\eta_0<1$ and a function $\epsilon_0:(0,\eta_0)\to (0,\infty)$ such that  $\lim_{\eta\to 0}\epsilon_0(\eta)=0$, and
\begin{align*}
    g\notin G_{\eta}^M \implies g\in \bigcup_{I\subset \{1,\cdots,r\},\lambda\in F} \Sigma_{I,\lambda,B}^M(\theta,\epsilon_0(\eta))\cdot \Gamma.
\end{align*}
\end{proposition}
For any $I\subset \{1,\cdots,r\}$, $i\in I$ and $\lambda\in F$, define $\alpha_i^{\lambda}(g):=\alpha_i(\lambda^{-1} g \lambda)$ for $g\in \lambda T\lambda^{-1}$, where $\alpha_i$ is the $i$-th simple root. Then we may assume that $\alpha_i^{\lambda}\in \Delta_{P_I^{\lambda}}$. For our purpose in this article, we further define the following set. For any $I\subset \{1,\cdots,r\}$, $i\in I$, $\lambda\in F$, bounded set $B\subset G$, and real numbers $\theta,\epsilon>0$, we denote 
\begin{align}\label{align: definition of Sigma when specifying a simple root}
    \Sigma^M_{I,i,\lambda,B}(\theta,\epsilon):=\{g\in G: g\in \Sigma^M_{I,\lambda,B}(\theta), \text{ and } \alpha_i^{\lambda}(a_g(I,\lambda))<\epsilon\}.
\end{align}
Note that by definition, we have
\begin{align}\label{align:old Sigma is union of new Sigma}
     \Sigma_{I,\lambda,B}^M(\theta,\epsilon)=\bigcup_{i\in I} \Sigma_{I,i,\lambda,B}^M(\theta,\epsilon).
\end{align}

Let $\mathfrak{v} \in \mathfrak{R}$ and 
$\boldsymbol{Q}$ be the maximal parabolic $\mathbb{Q}$-subgroup of $\boldsymbol{G}$ whose unipotent radical gives back $\mathfrak{v}$. We can find $i\in \{1,\cdots,r\}$, $\lambda\in F$, $\gamma\in \Gamma$ such that $\boldsymbol{Q}=\gamma \boldsymbol{P}_i^{\lambda}\gamma^{-1}$. 
 Fix $B,\theta,\eta_0,\epsilon_0$ as in Proposition \ref{Proposition: a description of cusps when taking consideration of M}. For $0<\eta<\eta_0$, define
\begin{align}\label{align: definition of U_mathfrak{v}}
    U_{\mathfrak{v}}^M(\epsilon_0(\eta)):=\bigcup_{I\subset \{1,\cdots,r\},i\in I}\Sigma^M_{I,i,\lambda,B}(\theta,\epsilon_0(\eta))\cdot \gamma^{-1}.
\end{align}
By Proposition \ref{Proposition: a description of cusps when taking consideration of M} and (\ref{align:old Sigma is union of new Sigma}), we obtain
\begin{corollary}\label{Corollary: a covering for group}
For any $0<\eta<\eta_0$, we have $G=G_{\eta}^M \cup \bigcup_{\mathfrak{u}\in \mathfrak{R}} U_{\mathfrak{u}}^M(\epsilon_0(\eta))$.
\end{corollary}

\begin{proposition}\label{Proposition: there is m in M such that Ad(mg)v is in W_0}
Let $\theta,\eta_0,\epsilon_0$ be as in Proposition \ref{Proposition: a description of cusps when taking consideration of M}. Given a bounded set $C\subset G$, for all sufficiently small $0<\eta<\eta_0$ (depending on C), the following holds: For any $i\in \{1,\cdots,r\}$, $\lambda\in F$ and $\gamma\in \Gamma$, let $\boldsymbol{Q}=\gamma \boldsymbol{P}_i^{\lambda}\gamma^{-1}$ be a maximal parabolic $\mathbb{Q}$-subgroup  and define $\mathfrak{v}:=\mathrm{Lie}(\mathrm{Rad}_{\mathrm{U}}(\boldsymbol{Q}))$. Given any $I\subset \{1,\cdots,r\}$ containing $i$, $g\in G$, and $m\in M$. 
Assume that $m$ and $g_1:=g\gamma$ satisfy
\begin{itemize}
    \item $g_1^{-1} M g_1 \subset P_I^{\lambda}$;
    \item  $\forall \alpha\in \Delta_{P_I^{\lambda}}$, $\alpha(a_{g_1}(I,\lambda))<\theta$;
    
    \item $\alpha_i^{\lambda}(a_{g_1}(I,\lambda))<\epsilon_0(\eta)$;
    
    \item $p_{g_1}(I,\lambda)\cdot g_1^{-1}m g_1\in (C\cap {^0\!P}_I^{\lambda})\cdot (\Gamma\cap {^0\!P}_I^{\lambda})$.
\end{itemize}
Then $\mathrm{Ad}(mg)\mathfrak{v}\subset \mathrm{Span}_{\mathbb{R}}(\mathfrak{g}_{mg}\cap W_0)$, where $W_0$ is a Zassenhaus neighborhood as in Proposition \ref{Proposition: existence of Zassenhauss neighborhood}.
\end{proposition}

The four conditions listed above say that $g_1 \in \Sigma^{M}_{I,i,\lambda,C}(\theta,\ep_0(\eta))$ and the $m$ implicit in the definition is exactly our $m$.

\begin{proof}
We write $\boldsymbol{P}=\boldsymbol{P}_I^{\lambda}$, $k_{g_1}=k_{g_1}(I,\lambda)$, $a_{g_1}=a_{g_1}(I,\lambda)$, and $p_{g_1}=p_{g_1}(I,\lambda)$ for short. By definition, for all $j\in I$, $\alpha_j(\lambda^{-1} a_{g_1} \lambda)<\theta<1$, and $\alpha_i(\lambda^{-1} a_{g_1} \lambda)<\epsilon_0(\eta)$. Choose a primitive integral basis $v_1,\cdots,v_{d_i}\in \mathfrak{g}_{\mathbb{Z}}$
of $\mathfrak{u}_i$, so  $\mathfrak{u}_i=\mathrm{Span}_{\mathbb{R}}\{v_j:1\leq j\leq d_i\}$. Since there are only finitely many standard maximal parabolic $\mathbb{Q}$-subgroups, $\norm{v_j}$ is bounded above uniformly for $1\leq j\leq d_i$.

Since $C$ is bounded and $F$ is a finite set, the operator norm of any element in $\cup_{\lambda \in F}\lambda^{-1} C \lambda$ on $\mathfrak{g}$ is bounded above uniformly.  By assumption, there is $\gamma_0\in \Gamma\cap {^0\!P}$ such that
\begin{align*}
    p_{g_1}\cdot g_1^{-1}mg_1\cdot \gamma_0=b \in C\cap {^0\!P}.
\end{align*}
Note that $\mathrm{Rad}_{\mathrm{U}}(P_i)$ is normal in $P_i$, and $\lambda^{-1} b \lambda\in P_I\subset P_i$.
There is some $c_1>0$ such that
for very $v_j$, there exist $c^j_1,\cdots,c^j_{d_i}\in \mathbb{R}$ such that $|c^j_k|<c_1$ for all $k$, and
\begin{align*}
    \mathrm{Ad}(\lambda^{-1} b \lambda)v_j=\sum_{k=1}^{d_i} c^j_{k} v_k.
\end{align*}
Therefore, for $j=1,\cdots,d_i$,
\begin{align*}
    \mathrm{Ad}(mg_1\gamma_0\lambda)v_j&=\mathrm{Ad}(g_1\cdot g_1^{-1}mg_1\gamma_0\lambda)v_j\\ \nonumber
    &=\mathrm{Ad}(k_{g_1} a_{g_1} p_{g_1} g_1^{-1}mg_1 \gamma_0\lambda)v_j\\ \nonumber
    &=\mathrm{Ad}(k_{g_1}\lambda)\mathrm{Ad}(\lambda^{-1} a_{g_1} \lambda )\mathrm{Ad}(\lambda^{-1} b\lambda)v_j\\ \nonumber
    &=\mathrm{Ad}(k_{g_1}\lambda)\mathrm{Ad}(\lambda^{-1} a_{g_1} \lambda )\sum_{k=1}^{d_i} c^j_{k} v_k. 
\end{align*}
By the description of the Lie algebra of $\boldsymbol{P}_i$ in (\ref{Align: Lie algebra decomposition of parabolic 1}) and (\ref{Align: Lie algebra decomposition of parabolic 2}), and the assumption of the Proposition, we have for each $k=1,\cdots,d_i$,
\begin{align*}
    \norm{\mathrm{Ad}(\lambda^{-1} a_{g_1} \lambda )v_k}<|\alpha_i(\lambda^{-1}a_{g_1}\lambda)|\norm{v_k}<\epsilon_0(\eta) \norm{v_k}.
\end{align*}
Let $\eta^{\prime}>0$ be such that the ball $W_{\eta^{\prime}}$ of radius $\eta^{\prime}$ centered at $0$ in $\mathfrak{g}$ satisfies $W_{\eta^{\prime}}\subset W_0$. By boundedness of $c^j_k$, compactness of $K$, and finiteness of $F$, choosing $\eta>0$ small enough (so $\epsilon_0(\eta)$ is small), we have for any $j=1,\cdots,d_i$,
\begin{align*}
    \norm{\mathrm{Ad}(mg_1\gamma_0\lambda)v_j}<\frac{\eta^{\prime}}{N},
\end{align*}
where $N<\infty$ is the smallest positive integer such that $\mathrm{Ad}(\lambda)N v\in \mathfrak{g}_{\mathbb{Z}}$ for any $v\in \mathfrak{g}_{\mathbb{Z}}$ and any $\lambda \in F$. Since $\gamma_0\in \Gamma$, by the choice of $N$, we have 
\begin{align*}
    &\mathrm{Ad}(\gamma_0 \lambda) Nv_j\in \mathfrak{g}_{\mathbb{Z}}, \text{ and }
    \norm{\mathrm{Ad}(mg_1 \gamma_0\lambda) N v_j}< \eta^{\prime},\text{ }  \forall j=1,\cdots,d_i.
\end{align*}
Therefore, $\mathrm{Ad}(m g_1\gamma_0 \lambda)N v_j\in \mathfrak{g}_{mg_1}\cap W_0$ for $1\leq j\leq d_i$.
Note that as $\gamma_0\in \lambda P_i \lambda^{-1}\cap \Gamma$, and $\mathfrak{u}_i$ is spanned by $v_j$'s, we then have
\begin{align*}
    \mathrm{Ad}(mg_1)\mathrm{Ad}(\lambda)\mathfrak{u}_i=\mathrm{Ad}(mg_1)\mathrm{Ad}(\gamma_0\lambda)\mathfrak{u}_i\subset \mathrm{Span}_{\mathbb{R}}(\mathfrak{g}_{mg_1}\cap W_0).
\end{align*}
Since $g_1=g\gamma$ and $\Gamma$ preserve $\mathfrak{g}_{\mathbb{Z}}$, we have $\mathfrak{g}_{mg_1}=\mathfrak{g}_{mg}$, and
\begin{align*}
  \mathrm{Ad}(mg)\mathfrak{v}=\mathrm{Ad}(mg)\mathrm{Ad}(\gamma \lambda)\mathfrak{u}_i=\mathrm{Ad}(mg_1)\mathrm{Ad}(\lambda)\mathfrak{u}_i \subset \mathrm{Span}_{\mathbb{R}}(\mathfrak{g}_{mg}\cap W_0).
\end{align*}
\end{proof}

\begin{corollary}\label{Corollary: If g in U_u^M, then Ad(mg)u is contained in span...}
Let $B,\theta,\eta_0,\epsilon_0$ be as in Proposition \ref{Proposition: a description of cusps when taking consideration of M}. For $i\in \{1,\cdots,r\}$, $\lambda\in F$ and $\gamma\in \Gamma$, let $\boldsymbol{Q}=\gamma \boldsymbol{P}_i^{\lambda}\gamma^{-1}$ be a maximal parabolic $\mathbb{Q}$-subgroup with $\mathfrak{v}=\mathrm{Lie}(\mathrm{Rad}_{\mathrm{U}}(\boldsymbol{Q}))$. Then for any $I\subset \{1,\cdots,r\}$ with $i\in I$, any $0<\eta<\eta_0$ sufficiently small, and any $g\in \Sigma^M_{I,i,\lambda,B}(\theta,\epsilon_0(\eta))\cdot \gamma^{-1}$, there exists $m\in M$ such that
\begin{align}\label{align: consequence of g in U_u^M}
   \mathrm{Ad}(mg)\mathfrak{v}\subset \mathrm{Span}_{\mathbb{R}}(\mathfrak{g}_{mg}\cap W_0),
\end{align}
where $W_0$ is a Zassenhaus neighborhood as in Proposition \ref{Proposition: existence of Zassenhauss neighborhood}. In particular, if $g\in U_{\mathfrak{v}}^M(\epsilon_0(\eta))$, then there exists $m\in M$ such that (\ref{align: consequence of g in U_u^M}) holds.
\end{corollary}
\begin{proof}
By assumption, $g\gamma \in \Sigma^M_{I,i,\lambda,B}(\theta,\epsilon_0(\eta))$. By definition of $\Sigma^M_{I,i,\lambda,B}(\theta,\epsilon_0(\eta))$, there is $m\in M$ such that $g_1=g\gamma$ and $m$ satisfy all the assumptions of Proposition \ref{Proposition: there is m in M such that Ad(mg)v is in W_0} with $B$ in place of $C$. Therefore, the corollary follows.
\end{proof}

The following proposition will be useful when we apply Theorem \ref{Theorem: a covering theorem} to a certain topological cover of the $\mathbb{R}$-diagonalizable subgroup $A$ constructed later on:
\begin{proposition}\label{Proposition: when M is nontrivial, if the intersection is nonempty, the span is again unipotent}
Given a positive integer $n\leq r$, where $r= \mathrm{rank}_{\mathbb{Q}}(\boldsymbol{G})$, there is a sufficiently small $\eta>0$ such that the following holds: For any integer $1\leq k\leq n$, and $k$ maximal parabolic $\mathbb{Q}$-subgroups $\boldsymbol{Q}_1,\cdots,\boldsymbol{Q}_k$ whose unipotent radicals have $\mathfrak{v}_1,\cdots,\mathfrak{v}_k$ as their Lie algebras, if $\bigcap_{i=1}^k U^M_{\mathfrak{v}_i}(\epsilon_0(\eta))\neq\emptyset$, then the Lie subalgebra generated by $\mathrm{Span}_{\mathbb{R}}\{\mathfrak{v}_i:i=1,\cdots,k\}$ is unipotent.
\end{proposition}
To prove Proposition \ref{Proposition: when M is nontrivial, if the intersection is nonempty, the span is again unipotent}, we invoke the following fundamental result of Dani and Margulis \cite{Dani_Margulis_1991_Asymptotic_behaviour_of_trajectories_of_unipotent_flows_on_homogeneous_spaces_MR1101994}:

\begin{theorem}\label{Theorem: quantitative nondivergence of unipotent orbits}\cite[Theorem 2]{Dani_Margulis_1991_Asymptotic_behaviour_of_trajectories_of_unipotent_flows_on_homogeneous_spaces_MR1101994}
Let $\boldsymbol{L}$ be a connected linear algebraic group defined over $\mathbb{Q}$ without nontrivial $\mathbb{Q}$-characters.
Let $\Gamma$ be an arithmetic subgroup of $L$. Then for any $\epsilon>0$ and any compact subset $B$ of $L/\Gamma$, there exists a compact set $C$ of $L/\Gamma$ such that for any unipotent one-parameter subgroup $\{u(t):t\in\mathbb{R}\}$ of $L$ and $g\in L$, if $g\Gamma/\Gamma\in B$, then for all large $T>0$, 
\begin{align*}
    \frac{1}{T}|\{t\in [0,T]:u(t)g\Gamma\in C\}|>1-\epsilon,
\end{align*}
where $|\cdot|$ denotes the Lebesgue measure of a measurable set.
\end{theorem}
\begin{proof}
Write $\boldsymbol{L}=\boldsymbol{H}\cdot \boldsymbol{R}$ where $\boldsymbol{H}$ is semisimple and $\boldsymbol{R}$ is the solvable radical of $\boldsymbol{L}$.
Both $\boldsymbol{H}$ and  $\boldsymbol{R}$ are defined over $\mathbb{Q}$.
By assumption the quotient of $\boldsymbol{R}$ by its unipotent radical is a $\mathbb{Q}$-anisotropic torus.
\cite[Theorem 2]{Dani_Margulis_1991_Asymptotic_behaviour_of_trajectories_of_unipotent_flows_on_homogeneous_spaces_MR1101994} implies that the above theorem holds for $\boldsymbol{\overline{L}}:=\boldsymbol{L}/\boldsymbol{R}$. 
Let $\pi:L\to L/R$ be the natural quotient map.
Since the natural projection map $L/\Gamma\to \pi(L)/\pi(\Gamma)$ is proper, we are done.
\end{proof}
We also require the following
\begin{lemma}\label{Lemma: intersection of certain subsets is nonempty}
Let $I\subset \mathbb{R}$ be a nonempty bounded (open or closed) interval, $k$ be a positive integer, and $I_1,\cdots,I_k$ be measurable subsets of $I$. If there exists $0<\epsilon<\frac{1}{k}$ such that $|I_i|>(1-\epsilon)|I|$ for each $i=1,\cdots,k$, then $\bigcap_{i=1}^k I_i\neq \emptyset$. Here $|\cdot|$ denotes the Lebesgue measure on $\mathbb{R}$.
\end{lemma}
\begin{proof}
Without loss of generality, we may assume that $|I|=1$. We will use induction to show that for any $1\leq j\leq k$, $|\bigcap_{i=1}^j I_i|>1-j\epsilon$. Since $0<\epsilon<\frac{1}{k}$, we have $|\bigcap_{i=1}^k I_i|>0$. In particular, $\bigcap_{i=1}^k I_i\neq \emptyset$.

When $j=1$, by assumption we have $|I_1|>1-\epsilon$. Suppose that for some $1\leq j\leq k-1$, $|\bigcap_{i=1}^j I_i|>1-j\epsilon$. Let $J_j=\bigcap_{i=1}^j I_i$. Then we have
\begin{align*}
    1\geq |J_j \cup I_{j+1}|&=|J_j|+|I_{j+1}|-|J_j\cap I_{j+1}|
    >1-j\epsilon +1-\epsilon-|J_j\cap I_{j+1}|.
\end{align*}
Therefore, $|\bigcap_{i=1}^{j+1}I_i|=|J_i\cap I_{j+1}|>1-(j+1)\epsilon$.
\end{proof}

\begin{proof}[Proof of Proposition \ref{Proposition: when M is nontrivial, if the intersection is nonempty, the span is again unipotent}]
Fix a positive number $\epsilon<\frac{1}{2r}$ and a compact neighborhood $C_1$ of $\mathrm{id}$ in $G$. 

Recall from (\ref{Align: finite collection of parabolic subgroups}) that $ \mathcal{B}=\{\lambda \boldsymbol{P}_I \lambda^{-1}: I\subset \Delta_{\mathbb{Q}},\lambda\in F\}$.
For any $\boldsymbol{P}\in \mathcal{B}$, denote $\Lambda_P=\Gamma\cap {^0\!P}$, and $B_P=(C_1\cdot B)\cap {^0\!P}$, 
where $B$ is the bounded set of $G$ as in Proposition \ref{Proposition: a description of cusps when taking consideration of M}. Since $\boldsymbol{^0\!P}$ has no nontrivial $\mathbb{Q}$-characters, $\Lambda_P$ is a lattice in $^0\!P$. We denote by $\pi_P:{^0\!P} \to {^0\!P}/\Lambda_P$ the natural projection map. Note that $\boldsymbol{^0\!P}=\boldsymbol{H}_P\cdot \boldsymbol{N}_P$, where $\boldsymbol{H}_P$ is a semisimple $\mathbb{Q}$-algebraic group and $\boldsymbol{N}_P$ is the unipotent radical of $\boldsymbol{^0\!P}$.

For each $\boldsymbol{P}\in \mathcal{B}$, applying Theorem \ref{Theorem: quantitative nondivergence of unipotent orbits} to $^0\!P / \Lambda_P$, $\epsilon$ and the compact set $B_P\Lambda_P/\Lambda_P$, we obtain a compact subset $C_P\subset {^0\!P}/\Lambda_P$ such that for any $p\in {^0\!P}$, any one-parameter unipotent subgroup $\{u_P(t):t\in\mathbb{R}\}$ of $^0\!P$, if $p\Lambda_P/\Lambda_P\in B_P\Lambda_P/\Lambda_P$, then for all large $T>0$,
\begin{align*}
    \frac{1}{T}|\{t\in [0,T]:u_P(t)p\Lambda_P/\Lambda_P\in C_P\}|>1-\frac{\epsilon}{2}.
\end{align*}
Since the cardinality of $\mathcal{B}$ is finite, we can choose a bounded set $C\subset G$ such that for any $\boldsymbol{P}\in \mathcal{B}$, the compact set $C_P$ obtained above satisfies $C_P\subset \pi_P(C\cap {^0\!P})$. We fix this bounded set $C$ for the rest of the proof, and choose $\eta>0$ to be sufficiently small such that Proposition \ref{Proposition: there is m in M such that Ad(mg)v is in W_0} holds.

Let $g\in \bigcap_{i=1}^k U^M_{\mathfrak{v}_i}(\epsilon_0(\eta))$. By definition of $U^M_{\mathfrak{v}_i}(\epsilon_0(\eta))$, for each $1\leq i\leq k$, there exist $j_i\in J_i\subset \{1,\cdots,r\}$, $\lambda_i\in F$, and $\gamma_i\in \Gamma$ such that $g\in\Sigma^M_{J_i,j_i,\lambda_i,B}(\theta,\epsilon_0(\eta))\cdot \gamma_i^{-1}$ and $\boldsymbol{Q}_i=\gamma_i \lambda_i \boldsymbol{P}_{j_i}\lambda_i^{-1}\gamma_i^{-1}$. To simplify our notations, for each $1\leq i\leq k$ we write $\boldsymbol{P}_i^{\prime}=\lambda_i \boldsymbol{P}_{j_i}\lambda_i^{-1}$, then $\boldsymbol{P}_i^{\prime}\in \mathcal{B}$. We also denote $B_i=B_{P_i^{\prime}}$, $C_i=C_{P_i^{\prime}}$, $\Lambda_i=\Gamma\cap {^0\!P_i^{\prime}}$, and $g_i=g\gamma_i$, for $1\leq i\leq k$.

By Howe-Moore ergodic theorem (see e.g. \cite{Zimmer_1984_Ergodic_theory_and_semisimple_groups_MR776417}), we can find a one-parameter unipotent subgroup $\{u(t):t\in \mathbb{R}\}\subset M$ such that \begin{align*}
    \overline{g^{-1}M g\Gamma/\Gamma}=\overline{g^{-1}\{u(t):t\in \mathbb{R}_+\}g\Gamma/\Gamma}.
\end{align*}
Since $\gamma_i^{-1}\overline{g^{-1}Mg\Gamma/\Gamma}=\overline{g_i^{-1}M g_i\Gamma/\Gamma}$, we have for each $1\leq i\leq k$,
\begin{align*}
    \overline{g_i^{-1}M g_i\Gamma/\Gamma}=\overline{g_i^{-1}\{u(t):t\in \mathbb{R}_+\}g_i\Gamma/\Gamma}.
\end{align*}
Note that for each $i$, $g_i^{-1}M g_i\subset {^0\!P_i^{\prime}}$.  Since $^0\!P_i^{\prime}\Gamma/\Gamma$ is closed, the natural map $^0\!P_i^{\prime}/\Lambda_i\to {^0\!P_i^{\prime}}\Gamma/\Gamma$ is proper (see e.g. \cite[Theorem 1.13]{Raghunathan_1972_Discrete_subgroups_of_Lie_groups_MR0507234}). Therefore, we have for each $1\leq i\leq k$,
\begin{align}\label{align: closure descends to smaller homogeneous spaces}
    \overline{g_i^{-1}M g_i\Lambda_i/\Lambda_i}=\overline{g_i^{-1}\{u(t):t\in \mathbb{R}_+\}g_i\Lambda_i/\Lambda_i}.
\end{align}

We may write $g_i=k_i a_i p_i$ with respect to the decomposition ${G=KP_i^{\prime}}$. 
Since $g_i\in \Sigma^M_{J_i,j_i,\lambda_i,B}(\theta,\epsilon_0(\eta))$, by definition we can find $m_i\in M$ such that  $p_i g_i^{-1} m_i g_i\in (B\cap {^0\!P^{\prime}_i})\cdot (\Gamma\cap {^0\!P_i^{\prime}})$. Depending on $g$ and $\boldsymbol{P}^{\prime}_i,i=1,\cdots,k$, we can choose $C_0$ to be a small enough neighborhood of $\mathrm{id}$ in $G$ such that for any $h\in C_0$ and $i=1,\cdots,k$, we have $p_i h p_i^{-1}\in C_1$. By (\ref{align: closure descends to smaller homogeneous spaces}), we find $t_i\geq 0$ and $h_i\in C_0\cap {^0\!P_i^{\prime}}$ such that 
\begin{align*}
    g_i^{-1} u(t_i) g_i \Lambda_i/\Lambda_i= h_i g_i^{-1} m_i g_i \Lambda_i/\Lambda_i.
\end{align*}
Then by the definition of $B_i$,
\begin{align*}
    p_i g_i^{-1} u(t_i) g_i \Lambda_i/\Lambda_i=p_i h_i p_i^{-1} p_i g_i^{-1} m_i g_i\Lambda_i/\Lambda_i\in  B_i \Lambda_i/\Lambda_i.
\end{align*}
Let $u_i(t)=p_i g_i^{-1} u(t) g_i p_i^{-1}$ be a one-parameter unipotent subgroup for $i=1,\cdots,k$. Then by the above, we have 
\begin{align*}
    u_i(t_i)p_i\Lambda_i/\Lambda_i\in   B_i \Lambda_i/\Lambda_i.
\end{align*}
By Theorem \ref{Theorem: quantitative nondivergence of unipotent orbits}, for any $i=1,\cdots,k$, for all large $T>0$,
\begin{align*}
    \frac{1}{T}|\{t\in [t_i,t_i+T]: u_i(t)p_i\Lambda_i\in C_i\}|>1-\frac{\epsilon}{2}.
\end{align*}
Without loss of generality, we may assume that $0\leq t_1\leq t_2\leq \cdots \leq t_k$. Then for any $i=1,\cdots,k$, we have for all large $T>0$,
\begin{align}\label{align: estimate of quantitative nondivergence of unipotent orbits}
    &\frac{1}{T+t_1-t_k}|\{t\in [t_k,t_1+T]:u_i(t)p_i\Lambda_i\in C_i\}|\nonumber \\
    &\geq \frac{T}{T+t_1-t_k}(\frac{1}{T}|\{t\in [t_i,t_i+T]:u_i(t)p_i\Lambda_i\in C_i\}|-\frac{t_k-t_1}{T})\nonumber\\
    &\geq \frac{T}{T+t_1-t_k}( 1-\frac{\epsilon}{2}-\frac{\epsilon}{10})\nonumber\\
    &\geq 1-\epsilon,
\end{align}
where we choose $T$ large enough such that $(t_k-t_1)/T<\epsilon/10$, and $\frac{T}{T+t_1-t_k}( 1-\frac{\epsilon}{2}-\frac{\epsilon}{10})\geq 1-\epsilon$.

Fix a large enough $T_0>0$ such that the estimate (\ref{align: estimate of quantitative nondivergence of unipotent orbits}) holds for each $i=1,\cdots,k$, and $t_1+T_0>t_k$. Let $I=[t_k,t_1+T_0]$ and
\begin{align*}
    I_i=\{t\in [t_k,t_1+T_0]:u_i(t)p_i\Lambda_i\in C_i\}.
\end{align*}
Applying Lemma \ref{Lemma: intersection of certain subsets is nonempty} to $I$,$I_1,\cdots,I_k$ and $\epsilon$, by the choice of $\epsilon$, we have $\bigcap_{i=1}^k I_i\neq \emptyset$. Let $t_0\in \bigcap_{i=1}^k I_i$. Then for $i=1,\cdots,k$,
\begin{align}\label{align: same u(t_0) bring it to a compact set}
    p_i g_i^{-1} u(t_0)g_i\Lambda_i \in C_i.
\end{align}
By definition of $g_i$, $C_i$ and (\ref{align: same u(t_0) bring it to a compact set}), it is clear that $g_i$ and $u(t_0)$ satisfies the assumptions of Proposition \ref{Proposition: there is m in M such that Ad(mg)v is in W_0}. By our choice of $\eta$, Proposition \ref{Proposition: there is m in M such that Ad(mg)v is in W_0} shows that for any $i=1,\cdots,k$, we have
\begin{align*}
    \mathrm{Ad}(u(t_0)g)\mathfrak{v}_i\subset \mathrm{Span}_{\mathbb{R}}(\mathfrak{g}_{u(t_0)g}\cap W_0).
\end{align*}
By Proposition \ref{Proposition: existence of Zassenhauss neighborhood}, the Lie algebra generated by $\mathrm{Span}_{\mathbb{R}}\{\mathfrak{v}_i:i=1,\cdots,k\}$ is unipotent.
\end{proof}

From now on, we fix a sufficiently small $\eta>0$ once and for all such that Proposition \ref{Proposition: there is m in M such that Ad(mg)v is in W_0}, Corollary \ref{Corollary: If g in U_u^M, then Ad(mg)u is contained in span...} and Proposition \ref{Proposition: when M is nontrivial, if the intersection is nonempty, the span is again unipotent} holds for this $\eta$. 

\begin{definition}\label{Definition: U_u^M}
For every $\mathfrak{u}\in\mathfrak{R}$, we write $U_{\mathfrak{u}}^M:=U_{\mathfrak{u}}^M(\epsilon_0(\eta))$ for short. As a reminder, $U_{\mathfrak{u}}^M(\epsilon_0(\eta))$ is defined in (\ref{align: definition of U_mathfrak{v}}).
\end{definition}

\subsection{A covering for torus}
In this subsection, we will let $\boldsymbol{G}$, $\boldsymbol{T}$, $\boldsymbol{S}$, $\boldsymbol{D}$, $A$, and $\boldsymbol{M}$ be as in Theorem \ref{Theorem: main theorem}.

For any $\mathfrak{u}\in \mathfrak{R}$, let $\boldsymbol{P}_{\mathfrak{u}}$ be the parabolic $\mathbb{Q}$-subgroup whose unipotent radical has Lie algebra $\mathfrak{u}$. For $\mathfrak{u}\in \mathfrak{R}$ and $g\in G$, define
\begin{align*}
    U_{\mathfrak{u}}^{A,g}:=\{a\in A:ag\in U_{\mathfrak{u}}^M\}.
\end{align*}
Note that if $g^{-1}\boldsymbol{M}g\not\subset \boldsymbol{P}_{\mathfrak{u}}$, then $ag\notin U^M_{\mathfrak{u}}$ for any $a\in A$, and so $U_{\mathfrak{u}}^{A,g}=\emptyset$. We also define 
\begin{align*}
    U_0^{A,g}:=\{a\in A: ag\in G_{\eta}^M\}.
\end{align*}

\begin{lemma}\label{Lemma: a covering for torus}
For any $g\in G$,
\begin{align*}
    A=U_0^{A,g}\cup \bigcup_{\mathfrak{u}\in \mathfrak{R}} U_{\mathfrak{u}}^{A,g}.
\end{align*}
Moreover, $\{U_0^{A,g}\}\cup \{U_{\mathfrak{u}}^{A,g}:\mathfrak{u}\in \mathfrak{R}\}$ is an open cover of $A$.
\end{lemma}
\begin{proof}
By Corollary \ref{Corollary: a covering for group}, we have
\begin{align*}
    Ag= Ag \cap G_{\eta}^M\cup \bigcup_{\mathfrak{u}\in \mathfrak{R}} Ag\cap U_{\mathfrak{u}}^M.
\end{align*}
Since $A\subset \mathrm{Z}_G(M)$, the assertion that this is an open cover follows by the definition of $U_0^{A,g}$ and $U_{\mathfrak{u}}^{A,g}$.
\end{proof}

As $\boldsymbol{D}$ is a maximal $\mathbb{R}$-split torus in $\boldsymbol{\mathrm{Z}_G(M)}$, we have 

\begin{proposition}\label{Proposition: consequence of M contained in parabolic subgroup}
Assume that $\boldsymbol{D}\subset \boldsymbol{T}$. Let $\tau$ be a Cartan involution of $\boldsymbol{G}$ such that for any $a\in T$, $\tau(a)=a^{-1}$. Let $g\in G$. Assume that there are $k$ standard maximal parabolic $\mathbb{Q}$-subgroups $\boldsymbol{P}_{i_1},\cdots,\boldsymbol{P}_{i_k}$ such that $g^{-1}\boldsymbol{M} g\subset \boldsymbol{P}_{i_j}$ for $1\leq j \leq k$. Then there exist $w\in \mathrm{W}(G)$, $h\in \mathrm{Z}_G(M)$ and $u\in \bigcap_{j=1}^k P_{i_j}$ such that
\begin{itemize}
    \item[(1)] g=hwu;
    
    \item[(2)] $w^{-1}\boldsymbol{M}w \subset \bigcap_{j=1}^{k} \boldsymbol{P_{i_j}}$;
    
    \item[(3)] $w^{-1}\boldsymbol{M}w \subset \bigcap_{j=1}^{k} \tau(\boldsymbol{P_{i_j}})$;
    
    \item[(4)] $\bigcap_{j=1}^k w\boldsymbol{P}_{i_j}w^{-1} \cap \boldsymbol{Z_G(M)}$ is a parabolic subgroup of $\boldsymbol{Z_G(M)}$;
    
    \item[(5)] $\{w(\chi_{i_j}):j=1,\cdots, k\}$ restricted to $D^{\circ}$ are linearly independent as linear functionals.
\end{itemize}
\end{proposition}
\begin{proof}
In the following proof, for any one-parameter subgroup $\{a(t)\}_{t\in \mathbb{R}}$ of $G$, we denote 
\begin{align*}
    \boldsymbol{P}_a:=\{g\in \boldsymbol{G}: \lim_{t\to +\infty} a(t)g a(-t) \text{ exists}\}; \text{ and }\\
    \boldsymbol{P}_{a^{-1}}:=\{g\in \boldsymbol{G}: \lim_{t\to +\infty} a(-t)g a(t) \text{ exists}\}.
\end{align*}
As $\boldsymbol{P}_{i_1},\cdots,\boldsymbol{P}_{i_k}$ are standard maximal parabolic $\mathbb{Q}$-subgroups, there are one-parameter subgroups $\{a_1(t):t\in \mathbb{R}\},\cdots,\{a_k(t):t\in \mathbb{R}\}$ of $S$ such that $\boldsymbol{P}_{i_j}=\boldsymbol{P}_{a_j}$, for $1\leq j\leq k$. Since $g^{-1} \boldsymbol{M} g\subset \bigcap_{j=1}^k \boldsymbol{P}_{i_j}$ and $\boldsymbol{M}$ is semisimple, by \cite[Proposition 11.23]{Borel_1991_Linear_algebraic_groups_MR1102012}, there exists $u_1\in \bigcap_{j=1}^k {P}_{i_j}$ such that $\{a_j(t):t\in \mathbb{R}\}\subset \mathrm{Z}_G(u_1 g^{-1} M g u_1^{-1})$ for $1\leq j \leq k$. Since $\boldsymbol{D}$ is a maximal $\mathbb{R}$-split torus of $\boldsymbol{\mathrm{Z}_G(M)}$, by \cite[Corollary 11.3]{Borel_1991_Linear_algebraic_groups_MR1102012}, we can find $h\in \mathrm{Z}_G(M)$ such that for all $1\leq j\leq k$,
\begin{align}
    \{h^{-1}g u_1^{-1}a_j(t)u_1 g^{-1} h:t\in \mathbb{R}\}\subset D.
\end{align}

For each $1\leq j \leq k$ and $t\in \mathbb{R}$, we denote $d_j(t):=h^{-1}g u_1^{-1} a_j(t) u_1 g^{-1} h\in D$. Define 
\begin{align*}
    \boldsymbol{D}^{\prime}=u_1 g^{-1}h \boldsymbol{D} h^{-1} g u_1^{-1},
     \boldsymbol{T}^{\prime}=u_1 g^{-1}h \boldsymbol{T} h^{-1} g u_1^{-1},
    \boldsymbol{M}^{\prime}=u_1 g^{-1}h \boldsymbol{M} h^{-1} g u_1^{-1}.
\end{align*}
Since for any $1\leq j\leq k$, $a_j(t)\in \mathrm{Z}_G(M^{\prime})$ and $a_j(t)\in D^{\prime}$, we have 
\begin{align}\label{align: M prime is contained in P}
   \boldsymbol{M}^{\prime}\subset \bigcap_{j=1}^{k} \boldsymbol{P_{i_j}},\text{ and }\boldsymbol{T}^{\prime}\subset \bigcap_{j=1}^{k} \boldsymbol{P_{i_j}}.
\end{align}
As $\boldsymbol{T}^{\prime}$ and $\boldsymbol{T}$ are maximal $\mathbb{R}$-split tori in $\bigcap_{j=1}^{k} \boldsymbol{P_{i_j}}$, there exists $u_2\in \bigcap_{j=1}^{k} {P_{i_j}}$ such that $u_2 \boldsymbol{T} u_2^{-1}=\boldsymbol{T}^{\prime}$. By definition of $\boldsymbol{T}^{\prime}$, we have $u_2^{-1}u_1 g^{-1} h \in \mathrm{N}_G(T)$. Therefore, there exists $w\in \mathrm{W}(G)$, or rather one of its representatives $w\in \mathrm{N}_G(T)$, such that $u_2^{-1}u_1 g^{-1} h=w^{-1}$. Thus, $g=hwu$ for $u=u_2^{-1}u_1$. This proves $(1)$.

Substituting $u_1 g^{-1} h$ by $u_2 w^{-1}$ in the definition of $\boldsymbol{M}^{\prime}$, by (\ref{align: M prime is contained in P}), (2) follows.

Define an involution $\tau^{\prime}$ of $\boldsymbol{G}$ by 
\begin{align*}
    \tau^{\prime}(g_1)=u_1 g^{-1}h\cdot \tau(h^{-1}g u_1^{-1}g_1 u_1 g^{-1}h)\cdot h^{-1}g u_1^{-1}, \forall g_1\in G.
\end{align*}
For any $1\leq j\leq k$, on one hand, using $u_1 g^{-1}h=u_2 w^{-1}$, we have 
\begin{align*}
     \tau^{\prime}(\boldsymbol{P}_{i_j})=u_2 w^{-1} \tau(w \boldsymbol{P}_{i_j}w^{-1})w u_2^{-1}
    =u_2 \tau(\boldsymbol{P}_{i_j}) u_2^{-1}.
\end{align*}
On the other hand, using the fact that $a_j(t)=u_1 g^{-1} h d_j(t)h^{-1} g u_1^{-1}$ and $d_j(t)\in T$, we have
\begin{align*}
    \tau^{\prime} (\boldsymbol{P}_{i_j})=u_1 g^{-1}h \tau(\boldsymbol{P}_{d_j}) h^{-1}g u_1^{-1}
    =\tau(\boldsymbol{P}_{i_j}).
\end{align*}
Since the normalizer of a parabolic subgroup is itself, by the above we obtain
\begin{align}\label{align: u_2 is contained in opposite parabolic}
    u_2\in \bigcap_{j=1}^k  \tau( {P}_{i_j}).
\end{align}
Note that for any $1\leq j\leq k$, as $a_j(t)\in \mathrm{Z}_G(M^{\prime})$, and $\tau(\boldsymbol{P}_{i_j})=\boldsymbol{P}_{a_j^{-1}}$, we have $\boldsymbol{M}^{\prime}\subset \bigcap_{j=1}^k  \tau( \boldsymbol{P}_{i_j} )$. Substituting $u_1 g^{-1} h$ by $u_2 w^{-1}$ in the definition of $\boldsymbol{M}^{\prime}$, by (\ref{align: u_2 is contained in opposite parabolic}) we obtain $(3)$.

To prove $(4)$, note that since $a_j(t)=u_2 w^{-1}d_j(t)w u_2^{-1}$, we have $\boldsymbol{P}_{d_j}=w \boldsymbol{P}_{i_j} w^{-1}$ for each $j$. As $d_j(t)\in D\subset \mathrm{Z}_G(M)$, for any $1\leq j \leq k$, $w\boldsymbol{P}_{i_j}w^{-1}\cap \boldsymbol{Z_G(M)}=\boldsymbol{P}_{d_j}\cap \boldsymbol{Z_G(M)}$ is a parabolic subgroup of $\boldsymbol{Z_G(M)}$. This proves $(4)$.

To see $(5)$, we use the expression $a_s(t)=u_2 w^{-1} d_s(t)w u_2^{-1}$ for any $1\leq s\leq k$ as follows: for any $1\leq j,s\leq k$, on one hand,
\begin{align*}
    \mathrm{Ad}(a_s(t))p_{i_j}=\chi_{i_j}(a_s(t))p_{i_j}.
\end{align*}
On the other hand,
\[
    \mathrm{Ad}(a_s(t))p_{i_j}=\mathrm{Ad}(u_2 w^{-1} d_s(t)w u_2^{-1})p_{i_j}=w(\chi_{i_j})(d_s(t))p_{i_j}.
\]
Therefore, we have $\chi_{i_j}(a_s(t))=w(\chi_{i_j})(d_s(t))$ for all $1\leq j,s\leq k$.
Since $\chi_{i_1},\cdots,\chi_{i_k}$ are linearly independent on the subgroup generated by  $\{a_j(t)\}_{t\in \mathbb{R}}$, $j=1,\cdots,k$,
we conclude that $w(\chi_{i_1}),\cdots,w(\chi_{i_k})$ are also linearly independent on the subgroup generated by $\{d_j(t)\}_{t\in \mathbb{R}}$, $j=1,\cdots,k$, and thus on $D^{\circ}$. This proves $(5)$.

\end{proof}

\begin{lemma}\label{Lemma: consequence of nonempty intersection}
Take $g\in G$ and $\mathfrak{v}_1,...,\mathfrak{v}_k\in \mathfrak{R}$.
If $\bigcap_{j=1}^k U_{\mathfrak{v}_j}^{A,g}$ is not empty, then there exist $g_0\in G$ and $\{i_1,\cdots,i_k\}\subset \{1,\cdots,r\}$ such that $\boldsymbol{P}_{\mathfrak{v}_j}=g_0 \boldsymbol{P}_{i_j} g_0^{-1}$, and
\begin{align*}
    \boldsymbol{M}\subset g g_0 \boldsymbol{P}_{i_j} g_0^{-1} g^{-1},\forall j=1,\cdots,k.
\end{align*}
\end{lemma}

\begin{proof}
Since $\bigcap_{j=1}^k U_{\mathfrak{v}_j}^{A,g}\neq \emptyset$, there exists $a\in A$ such that $ag\in \bigcap_{j=1}^k U_{\mathfrak{v}_j}^M$.
By definition of $U_{\mathfrak{v}_j}^M$ (see Definition \ref{Definition: U_u^M}), the fact that ${A}\subset {\mathrm{Z}_G(M)}$, and Proposition \ref{Proposition: when M is nontrivial, if the intersection is nonempty, the span is again unipotent}, we have
\begin{itemize}
    \item  $g^{-1} \boldsymbol{M} g\subset \boldsymbol{P}_{\mathfrak{v}_j}, \forall j=1,\cdots,k$, where $\boldsymbol{P}_{\mathfrak{v}_j}$ is the maximal parabolic $\mathbb{Q}$-subgroup of $\boldsymbol{G}$ whose unipotent radical has Lie algebra $\mathfrak{v}_j$.
    \item  the Lie subalgebra generated by $\mathrm{Span}_{\mathbb{R}}\{\mathfrak{v}_j:j=1,\cdots,k\}$ is unipotent.
\end{itemize}
By Proposition \ref{Proposition: if the span is unipotent, then there exist standard parabolic}, there exist $g_0\in G$ and $\{i_1,\cdots,i_k\}\subset \{1,\cdots,r\}$ such that 
\begin{align*}
    \mathfrak{v}_j=\mathrm{Ad}(g_0)\mathfrak{u}_{i_j},\forall j=1,\cdots,k.
\end{align*}
So we can write 
\begin{align*}
    \boldsymbol{P}_{\mathfrak{v}_j}=g_0 \boldsymbol{P}_{i_j} g_0^{-1},\forall j=1,\cdots,k.
\end{align*}
Therefore, $\boldsymbol{M}\subset g g_0 \boldsymbol{P}_{i_j} g_0^{-1} g^{-1}$ for any $ j=1,\cdots,k$.

\end{proof}

\begin{theorem}\label{Theorem: a bounded type Bruhat decomposotion}\cite[Theorem 2.1]{Solan_Tamam_2022_On_topologically_big_divergent_trajectories}
Let $\boldsymbol{L}$ be a connected 
reductive linear algebraic group over $\mathbb{R}$. Let $\boldsymbol{D}$ be a maximal $\mathbb{R}$-split torus of $\boldsymbol{L}$, $\boldsymbol{Q}_0$ be a $\mathbb{R}$-minimal parabolic subgroup of $\boldsymbol{L}$ containing $\boldsymbol{D}$, and $\boldsymbol{N}$ be a $\mathbb{R}$-maximal unipotent subgroup contained in $\boldsymbol{Q}_0$. Let $\mathrm{W}(L)\cong \mathrm{N}_L(D)/\mathrm{Z}_L(D)$ be the Weyl group of $\boldsymbol{L}$
and let $\widetilde{\mathrm{W}(L)}$ be a  set of representatives of $\mathrm{W}(L)$  in $\mathrm{N}_{L}(D)$.
Then there exist a compact set $N_0\subset N$ and $w_0\in \widetilde{\mathrm{W}(L)}$ such that $L=\widetilde{\mathrm{W}(L)} N_0 w_0 Q_0.$
\end{theorem}

\begin{proposition}\label{Proposition: consequence of a Bruhat type decomposition in linear 
representation}\cite[Corollary 2.2]{Solan_Tamam_2022_On_topologically_big_divergent_trajectories}
Let $\boldsymbol{L}$, $\boldsymbol{D}$, $\boldsymbol{Q}_0$ and $\mathrm{W}(L)$ be as in Theorem \ref{Theorem: a bounded type Bruhat decomposotion}. Let $\rho:\boldsymbol{L}\to \mathrm{GL}(V)$ be an $\mathbb{R}$-representation of  $\boldsymbol{L}$, and  $\norm{\cdot}$ be a fixed norm on $V$.
Let $\chi_0$ be a character of $\boldsymbol{D}$ and take $v_0 \in V_{\chi_0}$, the $\chi_0$-weight space.
We further assume that the line spanned by $v_0$ is stabilized by $\boldsymbol{Q_0}$.
Then there is $c=c(\norm{\cdot})>0$ such that for any $v_0\in V_{\chi_0}$ and any $l\in L$, there is $w\in \mathrm{W}(L)$ such that 
\begin{align*}
    \norm{\rho(l)v_0}\leq  c\norm{\pi_{w(\chi_0)}(\rho(l)v_0)},
\end{align*}
where $\pi_{w(\chi_0)}$ is the natural projection map to $w(\chi_0)$-weight space $V_{w(\chi_0)}$.
\end{proposition}

\begin{proof}
Applying Theorem \ref{Theorem: a bounded type Bruhat decomposotion}, for any $l\in L$, we can write $l=w_1 n_0 w_0 q$, where $w_1\in \mathrm{W}(L)$, $n_0\in N_0$ and $q\in Q_0$. By assumption, we have $\rho(q)v_0\in V_{\chi_0}$, and thus $\rho(w_0 q)v_0\in V_{w_0(\chi_0)}$.
As $N_0$ is compact and $\widetilde{\mathrm{W}(L)}$ is finite, there is a constant $c>0$ such that 
for every $n\in \widetilde{\mathrm{W}(L)} N_0 \widetilde{\mathrm{W}(L)}^{-1}$ and every $v\in V$ we have
\[
\norm{\rho(n)v} \leq c\norm{v}.
\]
Applying this to $w_1n_0w_1^{-1}$, we get
\begin{align*}
    \norm{\rho(l)v_0}=\norm{\rho(w_1 n_0 w_0 q)v_0}\leq c\norm{\rho(w_1 w_0 q)v_0}=c\norm{\pi_{w_1 w_0( \chi_0)}(\rho(l)v_0)},
\end{align*}
where the last equality comes from
\[
\pi_{w_1 w_0( \chi_0)}(\rho(l)v_0)
=w_1 \pi_{w_0( \chi_0)}(\rho(n_0w_0q)v_0) =
w_1 \pi_{w_0( \chi_0)}(\rho(w_0q)v_0).
\]
Setting $w=w_1 w_0$, the proposition follows.
\end{proof}

\begin{proposition}\label{Proposition: a locally finite cover}\cite[Proposition 4.2]{Solan_Tamam_2022_On_topologically_big_divergent_trajectories}
Assume that $\boldsymbol{D}\subset \boldsymbol{T}$. There exists a finite set $\Psi\subset \mathrm{X}(D)$ satisfying the following: For every $g\in G$ and $\mathfrak{u}\in \mathfrak{R}$, there exist a finite subset $\Psi_{\mathfrak{u}}^g\subset \Psi$ and a set of constants $\{d^g_{\mathfrak{u},\psi}\in \mathbb{R}:\psi\in \Psi_{\mathfrak{u}}^g\}$  such that:
\begin{itemize}
    \item[(1)] We have the inclusion
\begin{align*} U_{\mathfrak{u}}^{A,g}\subset  U_{\mathfrak{u},0}^{A,g}:=\left\{a\in A:\; \lambda(a)< d^g_{\mathfrak{u},\lambda},
\forall \lambda\in \Psi_{\mathfrak{u}}^g
\right\}.
\end{align*}

\item[(2)] The collection $\{U_{\mathfrak{u},0}^{A,g}:\mathfrak{u}\in \mathfrak{R}\}$ is locally finite.

\item[(3)] Take $\{\mathfrak{v}_j,\;j=1,\cdots, k\}\subset \mathfrak{R}$ and assume that $\bigcap_{j=1}^k U_{\mathfrak{v}_j}^{A,g}$ is not empty. Then there exist $w\in \mathrm{W}(G)$, $w^{\prime}\in \mathrm{W}(\mathrm{Z}_G(M))$, $\{\chi_{i_1},\cdots,\chi_{i_k}\}\subset \{\chi_i:i=1,\cdots,r\}$ and $\{c_j\in \mathbb{R}: j=1,\cdots k\}$ such that $w^{\prime}w(\chi_{i_1}),\cdots,w^{\prime}w(\chi_{i_k})$ are linearly independent as linear functionals on $D^{\circ}$, and \begin{align*}
    \bigcap_{j=1}^k U_{\mathfrak{v}_j}^{A,g}\subset \left\{a\in A:\; w^{\prime}w(\chi_{i_j})(a)<c_j,j=1,\cdots,k
    \right\}.
\end{align*}
\end{itemize}
\end{proposition}

\begin{proof}
Recall that $p_{\mathfrak{u}}$ is a primitive integral vector representing $\mathfrak{u}$.
By \cite[Corollary 3.3]{Kleinbock_Weiss_2013_Modified_Schmidt_games_and_a_conjecture_of_Margulis_MR3296561}, let $C_0>0$ be such that for any $g\in G$, any $\mathfrak{u}\in \mathfrak{R}$, if $\mathrm{Ad}(g)\mathfrak{u}\subset \mathrm{Span}_{\mathbb{R}}(\mathfrak{g}_{g}\cap W_0)$, then
\begin{align}\label{Align: consequence of being in small neighborhood of 0}
    \norm{\rho_j(g)p_{\mathfrak{u}}}< e^{-C_0}.
\end{align}
Recall from Section \ref{section: fundamental weights} that $\Phi_j$'s are the collection of weights of $T$ appearing in ``fundamental representations''.
Let $\Psi=\bigcup_j \Phi_j$.
For each $\mathfrak{u}\in \mathfrak{R}_j$, define
\begin{align*}
    \Psi^g_{\mathfrak{u}}:=\{\lambda\in \Phi_j: \pi_{\lambda}(\rho_j(g)p_{\mathfrak{u}})\neq 0\},
\end{align*}
where $\pi_{\lambda}$ is the natural projection to the weight subspace with weight $\lambda$. For any $\lambda\in \Psi^g_{\mathfrak{u}}$, denote
\begin{align*}
    d^g_{\mathfrak{u},\lambda}:=-\log(\norm{\pi_{\lambda}(\rho_j(g)p_{\mathfrak{u}})})-C_0.
\end{align*}
Let $\mathfrak{u}\in \mathfrak{R}_j$ and $a\in A$.
Assume $a\in U_{\mathfrak{u}}^{A,g}$, then $ag\in U_{\mathfrak{u}}^M$. By Corollary \ref{Corollary: If g in U_u^M, then Ad(mg)u is contained in span...}, there exists $m\in M$ such that $\mathrm{Ad}(mag)\mathfrak{u}\subset \mathrm{Span}_{\mathbb{R}}(\mathfrak{g}_{mag}\cap W_0)$. As $\boldsymbol{M}$ is semisimple and $g^{-1}\boldsymbol{M}g\subset \boldsymbol{P}_{\mathfrak{u}}$, we have $\rho_j(mag)p_{\mathfrak{u}}=\rho_j(ag)p_{\mathfrak{u}}$. Therefore, by (\ref{Align: consequence of being in small neighborhood of 0}),
\begin{align}\label{align: consequence of rho_j(ag)p_u being in a small neighborhood of 0}
    \norm{\rho_j(ag)p_{\mathfrak{u}}}<e^{-C_0}.
\end{align}
On the other hand, for every $\lambda \in \Psi^g_{\mathfrak{u}}$, we have 
\begin{align*}
    \norm{\rho_j(ag)p_{\mathfrak{u}}}&\geq \norm{\pi_{\lambda}(\rho_j(ag)p_{\mathfrak{u}})}= e^{\lambda(a)}\norm{\pi_{\lambda}(\rho_j(g))p_{\mathfrak{u}}}=e^{\lambda(a)-d^g_{\mathfrak{u},\lambda}-C_0}.
\end{align*}
With (\ref{align: consequence of rho_j(ag)p_u being in a small neighborhood of 0}), the above estimate shows that $\lambda(a)<d^g_{\mathfrak{u},\lambda}$. This finishes the proof of $(1)$.

Let us assume that $(2)$ is false. Then there are compact set $K\subset A$ and $\{\mathfrak{u}_i\in \mathfrak{R}: i\in \mathbb{N}\}$ such that $U_{\mathfrak{u}_i,0}^{A,g}\cap K\neq \emptyset$ for all $i\in \mathbb{N}$. By passing to a subsequence, we may assume that there is $1\leq j \leq r$ such that $\mathfrak{u}_i\in \mathfrak{R}_j$ for all $i\in \mathbb{N}$. By $(1)$ of the present proposition, for each $i$, we can write
\begin{align*}
    \rho_j(g)p_{\mathfrak{u}_i}=\sum_{\lambda\in \Psi_{\mathfrak{u}_i}^g} \pi_{\lambda}(\rho_j(g)p_{\mathfrak{u}_i}).
\end{align*}
For any $i\in \mathbb{N}$, there is $a_i\in K\cap U_{u_i,0}^{A,g}$, and we have
\begin{align}\label{align: estimate in proof of (2) in covering theorem}
    \rho_j(a_i g)p_{\mathfrak{u}_i}=\sum_{\lambda\in \Psi_{\mathfrak{u}_i}^g} e^{\lambda(a_i)}\pi_{\lambda}(\rho_j(g)p_{\mathfrak{u}_i}).
\end{align}
By definition of $U_{u_i,0}^{A,g}$, $\lambda(a_i)<-\log(\pi_{\lambda}(\rho_j(g)p_{\mathfrak{u}_i}))-C_0$ for all $\lambda\in \Psi_{\mathfrak{u}_i}^g$. Therefore, by (\ref{align: estimate in proof of (2) in covering theorem}), there exists $C_0^{\prime}>0$ such that for all $i\in \mathbb{N}$,
\begin{align*}
    \norm{ \rho_j(a_i g)p_{\mathfrak{u}_i}}<C_0^{\prime}.
\end{align*}
Since $K$ is compact, there exist $C_0^{\prime \prime}>C_0^{\prime}$ such that for all $i\in \mathbb{N}$,
\begin{align*}
    \norm{ \rho_j( g)p_{\mathfrak{u}_i}}<C_0^{\prime \prime},
\end{align*}
which is contrary to the discreteness of the set $\{\rho_j(g)p_{\mathfrak{u}_i}:i\in \mathbb{N}\}$. Hence $(2)$ holds.

Now let us prove (3). Assume that $\bigcap_{j=1}^k U_{\mathfrak{v}_j}^{A,g}$ is nonempty. By Lemma \ref{Lemma: consequence of nonempty intersection}, there exist $g_0\in G$ and $\{i_1,\cdots,i_k\}\subset \{1,\cdots,r\}$ such that $\boldsymbol{P}_{\mathfrak{v}_j}=g_0 \boldsymbol{P}_{i_j} g_0^{-1}$, and
\begin{align*}
    \boldsymbol{M}\subset g g_0 \boldsymbol{P}_{i_j} g_0^{-1} g^{-1},\forall j=1,\cdots,k.
\end{align*}
Applying Proposition \ref{Proposition: consequence of M contained in parabolic subgroup} with $g g_0$ in place of $g$ there, we find $w\in \mathrm{W}(G)$, $h\in \mathrm{Z}_G(M)$, and $u\in \bigcap_{j=1}^k\boldsymbol{P}_{i_j}$ such that
\begin{itemize}
    \item $g g_0=h w u$;
    \item $w^{-1}\boldsymbol{M}w \subset \bigcap_{j=1}^{k} \boldsymbol{P_{i_j}}$;
    \item $\bigcap_{j=1}^k w\boldsymbol{P}_{i_j}w^{-1} \cap \boldsymbol{Z_G(M)}$ is a parabolic subgroup of $\boldsymbol{Z_G(M)}$;
    
    \item $w(\chi_{i_1}),\cdots,w(\chi_{i_k})$ restricted to $D^{\circ}$ are linearly independent.
\end{itemize}

Take a $\mathbb{R}$-minimal parabolic subgroup $\boldsymbol{Q}_0$ of $\boldsymbol{\mathrm{Z}_G(M)}$ containing $\boldsymbol{D}$.
For each $1\leq j\leq k$,
applying Proposition \ref{Proposition: consequence of a Bruhat type decomposition in linear representation} to $\boldsymbol{\mathrm{Z}}_{\boldsymbol{G}}(\boldsymbol{M})$, $\boldsymbol{D}$, $\boldsymbol{Q}_{0}$, $\rho_{j}$, $\rho_{j}(w)p_{i_j}$ and $w(\chi_{i_j})$ in place of $\boldsymbol{L}$, $\boldsymbol{D}$, $\boldsymbol{Q}_0$, $\rho$, $v_0$, and $\chi_0$, we obtain a $c>0$ such that for any $h\in \mathrm{Z}_G(M)$, there is $w^{\prime}\in \mathrm{W}(\mathrm{Z}_G(M))$ such that
\begin{align*}
    \norm{\rho_{j}(hw)p_{i_j}}\leq c\norm{\pi_{w^{\prime} w(\chi_{i_j})}(\rho_{j}(hw)p_{i_j})}.
\end{align*}
Thus, we have
\begin{align*}
   0\neq\norm{\rho_{j}(g)p_{\mathfrak{v}_j}}&=\norm{\rho_j(g g_0)p_{i_j}}\\
        &=\norm{\rho_j(hwu)p_{i_j}}\\
        &\leq c\norm{\pi_{w^{\prime} w (\chi_{i_j})}(\rho_j(hw)p_{i_j})}\\
        &=c \norm{\pi_{w^{\prime} w (\chi_{i_j})}(\rho_j(g)p_{\mathfrak{v}_j})}.
\end{align*}
Therefore, we have $w^{\prime} w (\chi_{i_j})\in \Psi_{\mathfrak{v}_j}^g$, for $j=1,\cdots,k$. By (1) of Proposition \ref{Proposition: a locally finite cover}, there are real number $c_1,\cdots,c_k$ such that for any $a\in  U_{\mathfrak{v}_j}^{A,g}$, $w^{\prime}w(\chi_{i_j})(a)<c_j$ for $j=1,\cdots,k$. As a consequence, we have
\begin{align*}
    \bigcap_{j=1}^k U_{\mathfrak{v}_j}^{A,g}\subset \{a\in A: w^{\prime}w(\chi_{i_j})<c_j,j=1,\cdots,k\}.
\end{align*}
Since $w^{\prime}\in \mathrm{W}(\mathrm{Z}_G(M))$, and $w(\chi_{i_1}),\cdots,w(\chi_{i_k})$ are linearly independent on $D^{\circ}$, we conclude that $w^{\prime}w(\chi_{i_1}),\cdots,w^{\prime}w(\chi_{i_k})$ are also linearly independent on $D^{\circ}$.
\end{proof}

\section{Some linear algebra lemmas}\label{Section: Some linear algebra lemmas}
In this section, we will prove some simple, yet useful linear algebra lemmas, which will be crucial in the course of establishing Theorem \ref{Theorem: main theorem}.

In the following, $(\cdot|\cdot)$ denotes a strictly positive definite symmetric bilinear form in a real vector space $V$, with $\dim V=k$ for some $k\geq 1$. For any linear subspace $U\subset V$, we denote by $\pi_{U}$ the corresponding orthogonal projection map from $V$ to $U$, and $U^{\perp}$ the orthogonal complement to $U$ with respect to $(\cdot|\cdot)$.

\begin{lemma}\label{Lemma: a linear algebra lemma}
Let $v_1,\cdots,v_k\in V$ be $k$ linearly independent vectors. Let $\lambda_1,\cdots,\lambda_k$ be $k$ linear functionals on $V$ satisfying
\begin{align*}
    \lambda_i(v_j)=\delta_{ij}, \forall 1\leq i,j\leq k,
\end{align*}
where $\delta_{ij}$ equals $1$ if $i=j$ and zero otherwise. Furthermore, we define the finite set
\begin{align*}
    \Sigma:=\{\boldsymbol{\sigma}=(\sigma_1,\cdots,\sigma_k):\sigma_i=\pm 1,i=1,\cdots,k\},
\end{align*}
and for each $\boldsymbol{\sigma}\in \Sigma$,
\begin{align*}
    V_{\boldsymbol{\sigma}}:=\{v\in V: (\mathrm{sign}(\lambda_1(v)),\cdots,\mathrm{sign}(\lambda_k(v)))=\boldsymbol{\sigma}\}.
\end{align*}
Then for any choice of $v_{\boldsymbol{\sigma}}\in V_{\boldsymbol{\sigma}}$ for each $\boldsymbol{\sigma}\in \Sigma$, we have
$\mathrm{Span}_{\mathbb{R}}\{v_{\boldsymbol{\sigma}}:\boldsymbol{\sigma}\in \Sigma\}=V$.
\end{lemma}

\begin{proof}
Let $U=\mathrm{Span}_{\mathbb{R}}\{v_{\boldsymbol{\sigma}}:\boldsymbol{\sigma}\in \Sigma\}$. Suppose that $U\neq V$, then the orthogonal complement $U^{\perp}\neq 0$. Choose a nonzero $v\in U^{\perp}$. Since $v_1,\cdots,v_k$ are linearly independent, we can find $a_1,\cdots,a_k\in \mathbb{R}$ such that $v=\sum_{i=1}^k a_i v_i$. Similarly, for each $\boldsymbol{\sigma}\in \Sigma$, we find $b_1^{\boldsymbol{\sigma}},\cdots,b_k^{\boldsymbol{\sigma}}\in \mathbb{R}$ such that $v_{\boldsymbol{\sigma}}=\sum_{i=1}^k b_i^{\boldsymbol{\sigma}} v_i$. By the assumption on linear functionals $\lambda_i$'s, we have $(\mathrm{sign}(b_1^{\boldsymbol{\sigma}}),\cdots,\mathrm{sign}(b_k^{\boldsymbol{\sigma}}))=\boldsymbol{\sigma}$. Note that for any $\boldsymbol{\sigma}\in \Sigma$,
\begin{align}\label{align: contradiction to (v|v_{sigma})=0}
    (v|v_{\boldsymbol{\sigma}})=(\sum_{i=1}^k a_i v_i|\sum_{j=1}^k b_j^{\boldsymbol{\sigma}} v_j)=\sum_{j=1}^k b_j^{\boldsymbol{\sigma}}\sum_{i=1}^k a_i(v_i| v_j).
\end{align}
Since $v_1,\cdots,v_k$ are linearly independent, and $(\cdot|\cdot)$ is a strictly positive definite symmetric bilinear form, the matrix $((v_i|v_j))_{1\leq i,j\leq k}$ is nonsingular. As $v$ is nonzero, at least one of $a_i$'s is nonzero. Therefore, 
\begin{align*}
    (a_1,\cdots,a_k)\cdot((v_i|v_j))_{1\leq i,j\leq k}=(\sum_{i=1}^k a_i(v_i|v_1),\cdots,\sum_{i=1}^k a_i(v_i|v_k))\neq 0.
\end{align*}
Thus we can choose $\boldsymbol{\sigma}\in \Sigma$ such that (\ref{align: contradiction to (v|v_{sigma})=0}) is nonzero, contrary to $(v|v_{\boldsymbol{\sigma}})=0$. This shows that $U^{\perp}=0$ and proves the lemma.
\end{proof}

\begin{proposition}\label{Proposition: consequence of A being a proper subspace}
Let $U$ be a proper linear subspace of $V$. 
Let $\{\lambda_i\}_{i=1,...,k}$ be linearly independent  linear functionals on $V$.
Then there exists $v\in V$ such that for any $N>0$, and any $u\in U$, there exists $\lambda_{j(u)}$ for some $1\leq j(u) \leq k$ such that
\begin{align*}
    |\lambda_{j(u)}(u+N\cdot v)|>N.
\end{align*}
\end{proposition}
\begin{proof}
 Find $\{v_i\}_{i=1,...,k}$  satisfying $\lambda_i(v_j)=\delta_{ij}$ as in Lemma \ref{Lemma: a linear algebra lemma}, so they are linearly independent.
Also let $\Sigma$ and $\{V_{\boldsymbol{\sigma}}:\boldsymbol{\sigma}\in \Sigma\}$ be as in Lemma \ref{Lemma: a linear algebra lemma}. Since $U$ is proper, by Lemma \ref{Lemma: a linear algebra lemma}, there exists $\boldsymbol{\sigma}_0\in \Sigma$ such that $U\cap V_{\boldsymbol{\sigma}_0}=\emptyset$. Choose $v\in V_{\boldsymbol{\sigma}_0}$ such that $\min_{j=1,\cdots,k}|\lambda_j(v)|>1.$

Since $U\cap V_{\boldsymbol{\sigma}_0}=\emptyset$, for any $u\in U$, there exists $1\leq j(u)\leq k$ such that $\mathrm{sign}(\lambda_{j(u)}(u))\neq -\mathrm{sign}(\lambda_{j(u)}(v))$, where by convention we set $\mathrm{sign}(0)=0$. Hence, either $\lambda_{j(u)}(u)=0$, or $\lambda_{j(u)}(u)$ has the same sign as $\lambda_{j(u)}(v)$ does. Therefore, by the choice of $v$, we obtain
\begin{align*}
     |\lambda_{j(u)}(u+N\cdot v)|>N.
\end{align*}
\end{proof}

\begin{lemma}\label{Lemma: projection is onto if...}
Let $U$, $W$ be two linear subspaces of $V$. Assume that $\dim U=m\leq \dim W$. Then $\dim \pi_W(U)=\dim U$ if and only if $\pi_U(W)=U$.
\end{lemma}
\begin{proof}

Assume that $\dim \pi_W(U)=\dim U$, then we may choose $m$ linearly independent vectors $u_1,\cdots,u_m\in U$ such that $\pi_W(u_1),\cdots,\pi_W(u_m)$ are linearly independent. Suppose that $\pi_U(\pi_W(u_1)),\cdots,\pi_U(\pi_W(u_m))$ are not linearly independent, then we can find $a_1,\cdots,a_m\in \mathbb{R}$ such that at least one of them is nonzero, and $\pi_U(\pi_W(\sum_{i=1}^m a_i u_i))=0$. Therefore, $0\neq\pi_W(\sum_{i=1}^m a_i u_i)\in U^{\perp}$. Write $u=\sum_{i=1}^m a_i u_i$. By the decomposition $u=\pi_W(u)+\pi_{W^{\perp}}(u)$, we have 
\begin{align*}
    0=(u|\pi_W(u))=(\pi_W(u)+\pi_{W^{\perp}}(u)|\pi_W(u))=(\pi_W(u)|\pi_W(u)),
\end{align*}
which implies that $\pi_W(u)=0$, hence leads to a contradiction. Therefore, $\pi_U(\pi_W(u_1))$, $\cdots$, $\pi_U(\pi_W(u_m))$ are linearly independent, and so $\pi_U(W)=U$.

Conversely, assume that $\pi_U(W)=U$. We can choose $m$ linearly independent vectors $w_1,\cdots,w_m\in W$ such that $\mathrm{Span}_{\mathbb{R}}\{\pi_U(w_1),\cdots,\pi_U(w_m)\}=U$. By the same argument as above, we can show that $\pi_W(\pi_U(w_1)),\cdots,\pi_W(\pi_U(w_m))$ are  linearly independent, which implies that $\dim \pi_W(U)=\dim(U)$.
\end{proof}

\begin{corollary}\label{Corollary: If linear functionals are linearly independent, then ...}
Let $\lambda_1,\cdots,\lambda_m$ be $m$ linearly independent linear functionals on $V$. Let $u_1,\cdots,u_m\in V$ be such that $\lambda_i(v)=(u_i|v)$ for any $v\in V$. Denote $U=\mathrm{Span}_{\mathbb{R}}\{u_1,\cdots,u_m\}$. Let $W\subset V$ be a linear subspace. Then $\lambda_1,\cdots,\lambda_m$ restricted to $W$ are linearly independent if and only if $\pi_U(W)=U$. 
\end{corollary}
\begin{proof}
First we note that $\lambda_1,\cdots \lambda_m$ are linearly independent on $W$ if and only if $\pi_W(u_1)$, $\cdots$, $\pi_W(u_m)$ are linearly independent.

Suppose that $\lambda_1,\cdots \lambda_m$ are linearly independent on $W$, then $\pi_W(u_1)$, $\cdots$, $\pi_W(u_m)$ are linearly independent. By Lemma \ref{Lemma: projection is onto if...}, $\pi_U(W)=U$.

Conversely, suppose that $\pi_U(W)=U$, again by Lemma \ref{Lemma: projection is onto if...}, $\dim \pi_W(U)=\dim U$. Therefore, $\pi_W(u_1),\cdots,\pi_W(u_m)$ are linearly independent, so are $\lambda_1$, $\cdots$, $\lambda_m$ restricted to $W$.
\end{proof}

\begin{lemma}\label{Lemma: integral matrix has integral kernel}
Let $\lambda_1,\cdots,\lambda_m$ be $m$ linear functionals on $V$ (not necessarily linearly independent). Let $u_1,\cdots,u_k\in V$ be such that $\lambda_i(u_j)\in \mathbb{Z}$ for any $1\leq i\leq m$, $1\leq j\leq k$. Denote $U=\mathrm{Span}_{\mathbb{R}}(u_1,\cdots,u_k)$. Suppose that there exists $(a_1,\cdots,a_m)\in \mathbb{R}^m\setminus \{0\}$ such that $\sum_{i=1}^m a_i\lambda_i\equiv 0$ when restricted to $U$. Then there exists $(b_1,\cdots,b_m)\in \mathbb{Z}^m\setminus \{0\}$ such that $\sum_{i=1}^m b_i\lambda_i\equiv 0$ when restricted to $U$.
\end{lemma}
\begin{proof}
Consider the $m$ by $k$ matrix
\[
C=(\lambda_i(u_j))_{1\leq i\leq m, 1\leq j\leq k}.
\]
Since $(a_1,\cdots,a_m)\neq 0$, and $(a_1,\cdots,a_m)\cdot C=0$, $\mathrm{Ker}(C)\neq 0$. It is well known that the kernel of an integral matrix is spanned by integral vectors. Therefore, there exists nonzero $(b_1,\cdots,b_m)\in \mathbb{Z}^m\cap \mathrm{Ker}(C)$.
\end{proof}

\begin{corollary}\label{Corollary: linear dependency in algebraic torus}
Let $\boldsymbol{A}$ be an algebraic $\mathbb{R}$-split torus. Let $\lambda_1,\cdots,\lambda_m$ be $\mathbb{R}$-algebraic characters on $\boldsymbol{A}$. Suppose that there exists $(a_1,\cdots,a_m)\in \mathbb{R}^m\setminus\{0\}$ such that $\sum_{i=1}^m a_i\lambda_i\equiv 0$ on $\boldsymbol{A}$, then there exists $(b_1,\cdots,b_m)\in \mathbb{Z}^m\setminus\{0\}$ such that $\sum_{i=1}^m b_i \lambda_i\equiv 0$ on $\boldsymbol{A}$.
\end{corollary}
\begin{proof}
As $\boldsymbol{A}$ is an algebraic $\mathbb{R}$-split torus and $\lambda_1,\cdots,\lambda_m$ are $\mathbb{R}$-algebraic characters on $\boldsymbol{A}$, there exist $u_1,\cdots,u_n\in \mathrm{Lie}(A)$ such that $\lambda_i(u_j)\in \mathbb{Z}$ and $Span_{\mathbb{R}}(u_1,\cdots,u_n)=\mathrm{Lie}(A)$, see e.g. \cite[Proposition 8.6]{Borel_1991_Linear_algebraic_groups_MR1102012}. By Lemma \ref{Lemma: integral matrix has integral kernel}, the corollary follows.
\end{proof}

\section{Proof of Theorem \ref{Theorem: main theorem}}\label{Section: Proof of the main theorem}

\begin{proof}[Proof of $(1)\implies (2)$]
Let $n=\dim A$. Assume that $(2)$ does not hold.

Let $m$ be the maximal integer $k$ such that there exist $w\in \mathrm{W}(G)$ and $\{i_1,\cdots,i_k\}\subset \{1,\cdots,r\}$ such that  $w^{-1}\boldsymbol{M} w\subset \bigcap_{j=1}^k\boldsymbol{P}_{i_j}\cap \bigcap_{j=1}^k \tau(\boldsymbol{P}_{i_j})$, and $\{w(\chi_{i_j}):j=1,\cdots,k\}$ are linearly independent as linear functionals on $D$. Note that $m$ can be attained. Since $(2)$ does not hold, $m\leq n$.

Recall that we have fixed an $\eta>0$ once and for all such that Proposition \ref{Proposition: there is m in M such that Ad(mg)v is in W_0}, Corollary \ref{Corollary: If g in U_u^M, then Ad(mg)u is contained in span...} and Proposition \ref{Proposition: when M is nontrivial, if the intersection is nonempty, the span is again unipotent} hold for this $\eta$. Suppose that $Hg\Gamma\cap X_{\eta}=\emptyset$, then $U_0^{A,g}=\emptyset$. By Lemma \ref{Lemma: a covering for torus}, $\{U_{\mathfrak{u}}^{A,g}:\mathfrak{u}\in \mathfrak{R}\}$ forms an open cover of $A$.

Take $\{\mathfrak{v}_1,...,\mathfrak{v}_k\} \subset \mathfrak{R}$.
By Proposition \ref{Proposition: a locally finite cover} (3) and the negation of (2) in Theorem \ref{Theorem: main theorem}, if $k\leq n$ and $\bigcap_{j=1}^k U_{\mathfrak{v}_j}^{A,g}$ is nonempty, then there exist $k$ linearly independent linear functionals $\lambda_1,\cdots,\lambda_k$ on $A$ and real numbers $c_1,\cdots,c_k$ such that
\begin{align*}
    \mathrm{conv}(\bigcap_{j=1}^k U_{\mathfrak{v}_j}^{A,g})\subset \{a\in A: \lambda_j(a)<c_j,\forall j=1,\cdots,k\}.
\end{align*}
By Lemma \ref{Lemma: invariance dimension of convex set determined by k linearly independent functionals} and Lemma \ref{Lemma: inequality of invariance dimension}, the above implies that
\begin{align*}
    \text{invdim {conv}}(\bigcap_{j=1}^k U_{\mathfrak{v}_j}^{A,g})\leq n-k.
\end{align*}
Also, by Proposition \ref{Proposition: a locally finite cover} (2), the cover $\{\mathrm{conv}(U_{\mathfrak{u}}^{A,g}):\mathfrak{u}\in \mathfrak{R}\}$ is locally finite. Therefore, the open cover $\{U_{\mathfrak{u}}^{A,g}:\mathfrak{u}\in \mathfrak{R}\}$ of $A$ meets the assumptions of Theorem \ref{Theorem: a covering theorem}. Applying Theorem \ref{Theorem: a covering theorem}, we obtain $n+1$ different $\mathfrak{v}_1,\cdots,\mathfrak{v}_{n+1}$ in $\mathfrak{R}$ such that 
\begin{align*}
    \bigcap_{j=1}^{n+1} U_{\mathfrak{v}_j}^{A,g}\neq \emptyset.
\end{align*}
Then by Proposition \ref{Proposition: consequence of M contained in parabolic subgroup} and Lemma \ref{Lemma: consequence of nonempty intersection}, 
we have $m\geq n+1$, which leads to a contradiction. Therefore, $U_0^{A,g}$ is nonempty, and hence the $H$-action is uniformly non-divergent.
\end{proof}

The following proof can be regarded as a generalization of the phenomenon in \cite[Example 1]{Tomanov_Weiss_2003_Closed_orbits_for_actions_of_maximal_tori_on_homogeneous_spaces_MR1997950} (see also \cite[Section 9]{Solan_Tamam_2022_On_topologically_big_divergent_trajectories}).

\begin{proof}[Proof of $(2)\implies (3)$]
Assume that $(2)$ holds. Then there exist $w\in \mathrm{W}(G)$, $w^{\prime}\in \mathrm{W}(\mathrm{Z}_G(M))$ and $\{i_1,\cdots,i_k\}\subset \{1,\cdots,r\}$ such that
\begin{itemize}
    \item $w^{-1}\boldsymbol{M} w\subset \bigcap_{j=1}^k \boldsymbol{P}_{i_j} $;
    
    \item $w^{-1}\boldsymbol{M} w\subset \bigcap_{j=1}^k \tau(\boldsymbol{P}_{i_j}) $;
    
    \item $w^{\prime}w(\chi_{i_1}),\cdots,w^{\prime}w(\chi_{i_k})$ are not linearly independent as linear functionals on $A^{\circ}$.
\end{itemize}

Denote by $(\cdot|\cdot)$ the Killing form on $\mathrm{Lie}(G)$, which is a strictly positive definite symmetric bilinear form on $\mathrm{Lie} (T)$ (see e.g. \cite{Helgason_1978_Differential_geometry_Lie_groups_and_symmetric_spaces_MR514561}).

Let $u_1,\cdots, u_k \in\mathrm{Lie}(T)$ be such that $w(\chi_{i_j})(v)=(u_j|v)$ for every $v\in\mathrm{Lie}(T)$ and every $j=1,...,k$. 
Let $U:=\mathrm{Span}_{\mathbb{R}}\{u_1,\cdots,u_k\}$ and $\pi_U$ (resp. $\pi_{U^{\perp}}$) be the orthogonal projection from $\mathrm{Lie}(T)$ to $U$ (resp. $U^{\perp}$) with respect to $(\cdot|\cdot)$. By assumption, $w(\chi_{i_1}),\cdots,w(\chi_{i_k})$ are not linearly independent on $\mathrm{Ad}(w^{\prime -1})\mathrm{Lie}(A)$. By Corollary \ref{Corollary: If linear functionals are linearly independent, then ...}, $U^{\prime}:=\pi_U(\mathrm{Ad}(w^{\prime -1})\mathrm{Lie}(A))$ is a proper linear subspace of $U$. We also denote $U^{\perp \prime }:=\pi_{U^{\perp}}(\mathrm{Ad}(w^{\prime -1})\mathrm{Lie}(A))$.

Applying Proposition \ref{Proposition: consequence of A being a proper subspace} to $U$, $U^{\prime}$, and $w(\chi_{i_j})$ in place of $V$, $U$, and $\lambda_j$  there, we obtain $v\in U$ such that for any $N>0$, any $u\in U^{\prime}$, there is  some $1\leq j(u) \leq k$ such that $|w(\chi_{i_{j(u)}})(u+Nv)|>N$. Therefore,
\begin{align}\label{align: two inequalities in the proof of (2) implying (1) in the main theorem}
    \text{either } w(\chi_{i_{j(u)}})(u+Nv)<-N, \text{ or } -w(\chi_{i_{j(u)}})(u+Nv)<-N.
\end{align}
Recall that for each $1\leq i \leq r$, $p_i\in \bigwedge^{d_i}\mathfrak{g}$ (resp. $p_i^{-}\in \bigwedge^{d_i}\mathfrak{g}$) 
is the representative of the Lie algebra of $\mathrm{Rad}_{\mathrm{U}}(\boldsymbol{P}_i)$ (resp. $\mathrm{Rad}_{\mathrm{U}}(\tau(\boldsymbol{P}_i))$). As $w^{-1} \boldsymbol{M} w\subset \bigcap_{j=1}^k\boldsymbol{P}_{i_j}$, for $j=1,\cdots,k$, we have 
\begin{align}\label{align: equality 1 in proof of (2) implying (1) of main theorem}
    \mathrm{Ad}(H w^{\prime} \exp(Nv)w) p_{i_j}&=\mathrm{Ad}(A M w^{\prime} \exp(Nv)w)p_{i_j}\nonumber\\
    &=\mathrm{Ad}(w^{\prime}w^{\prime -1} A w^{\prime} \exp(Nv)w)p_{i_j}\nonumber\\
    &=\mathrm{Ad}(w^{\prime}) \mathrm{Ad}(\exp(U^{\prime}+U^{\perp\prime}+N v))\mathrm{Ad}(w)p_{i_j}\nonumber\\
    &=\mathrm{Ad}(w^{\prime})\mathrm{Ad}(\exp(U^{\prime}+Nv))\mathrm{Ad}(w)p_{i_j}\nonumber\\
    &=\exp(w(\chi_{i_j})(U^{\prime}+N v))\mathrm{Ad}(w^{\prime}w)p_{i_j},
\end{align}
where  for the second equality 
we use $w^{-1}\boldsymbol{M} w\subset \bigcap_{j=1}^k \boldsymbol{P}_{i_j} $.
And the third and fourth equality follow from the fact that for any $u\in U^{ \perp \prime}$, $w(\chi_{i_j})(u)=0$. Similarly, as $w^{-1} \boldsymbol{M} w\subset \bigcap_{j=1}^k\tau(\boldsymbol{P}_{i_j})$, for any $1\leq j\leq k$,
\begin{align}\label{align: equality 2 in proof of (2) implying (1) of main theorem}
    \mathrm{Ad}(H w^{\prime} \exp(Nv)w)  p_{i_j}^-=\exp(-w(\chi_{i_j})(U^{\prime}+Nv))\mathrm{Ad}(w^{\prime}w)p_{i_j}^-,
\end{align}
where in the above equality, we use the fact that $\mathrm{Ad}(a)p^-_{i_j}=-\chi_{i_j}(a)p^-_{i_j}$ for any $a\in T$.
Then by (\ref{align: two inequalities in the proof of (2) implying (1) in the main theorem}), for any $\epsilon>0$, there exists $N>0$ such that for any $h\in H$, there exists $1\leq j\leq k$ such that 
\begin{align}\label{align: either p_i is small or p_i^- is small}
   \text{either } \norm{\mathrm{Ad}(hw^{\prime} \exp(Nv)w)  p_{i_j}}<\epsilon, \text{ or }\norm{\mathrm{Ad}(hw^{\prime} \exp(Nv)w)  p_{i_j}^-}<\epsilon.
\end{align}
Since $p_{i_j}$ and $p^-_{i_j}$ are both nonzero $\mathbb{Q}$-vectors for any $1\leq j \leq k$, by (\ref{align: either p_i is small or p_i^- is small}), we conclude that $(3)$ holds.
\end{proof}

\begin{proof}[Proof of $(3)\implies (1)$]
This follows from Proposition \ref{Proposition: Not uniformly non-divergent criterion}.
\end{proof}

\section{Nondivergence in real rank one quotient}\label{Section: nonarithmetic quotient}
Throughout this section, let $\boldsymbol{G}$ be a connected semisimple $\mathbb{R}$-algebraic group with real rank one, and $\Gamma$ be an arbitrary lattice of $G$. By the Margulis arithmeticity theorem (see e.g. \cite{Zimmer_1984_Ergodic_theory_and_semisimple_groups_MR776417}), it is possible that $\Gamma$ of $G$ is non-arithmetic. 

Let $\mathfrak{g}=\mathfrak{k}\oplus\mathfrak{p}$ be a Cartan decomposition, where $\mathfrak{k}$ (resp. $\mathfrak{p}$) is the eigenspace with eigenvalue $1$ (resp. $-1$) of the corresponding Cartan involution. By \cite[Theorem 4.6]{Garland_Raghunathan_1970_Fundamental_domains_for_lattices_MR267041}, there are only finitely many unit vectors $Y$ in $\mathfrak{p}$ such that the unstable horosphere $N_{Y}$ of $\exp(Y)$ satisfies that $N_Y/N_Y\cap \Gamma$ is compact. If we fix such a $Y_0$, then for any other such $Y$, there exists $b_Y\in K$ such that $Y=Ad(b_Y^{-1})Y_0$. Let $\Xi$ be the collection of such $b_Y\in K$. In particular, the neutral element $e\in \Xi$. Let $\mathfrak{a}_{Y_0}$ be the $\mathbb{R}$-span of $Y_0$, and $A$ be the analytic subgroup corresponding to $\mathfrak{a}_{Y_0}$. Then there is a unique character (simple root) $\alpha$ of $A$ such that 
\[\mathfrak{g}=\mathfrak{g}_{-2\alpha}\oplus\mathfrak{g}_{-\alpha}\oplus \mathfrak{z}(\mathfrak{a}_{Y_0})\oplus \mathfrak{g}_{\alpha}\oplus \mathfrak{g}_{2\alpha},\]
where 
\[\mathfrak{g}_{i\alpha}:=\{v\in \mathfrak{g}:Ad(a)v=\exp(i\alpha(a))v,\forall a\in A\}, \quad i=\pm 1, \pm 2,\]
and $\mathfrak{Z}(\mathfrak{a}_{Y_0})$ is the centralizer of $\mathfrak{a}_{Y_0}$ in $\mathfrak{g}$.
Note that as before, by abuse of notations, for $a\in A$, \[\alpha(a):=\alpha(v), \quad \text{where }a=\exp(v) \text{ for }v\in \mathfrak{a}_{Y_0}.\]

Consider the Iwasawa decomposition $\boldsymbol{G}=\boldsymbol{KAN}$. Then $\boldsymbol{Z_G(A)}=\boldsymbol{M}\boldsymbol{A}$, where $\boldsymbol{M}=\boldsymbol{Z_G(A)}\cap \boldsymbol{K}$. Using these notations, we note that the minimal $\mathbb{R}$-parabolic subgroup $\boldsymbol{P}=\boldsymbol{MAN}$. Denote $^{\circ}\!\boldsymbol{P}=\boldsymbol{MN}$.

For any $t\in \mathbb{R}$, let
\[A_t:=\{a\in A: \alpha(a)<t\}.\]
As $G/\Gamma$ is not compact, we have the following theorem about fundamental domain of $G/\Gamma$:

\begin{theorem}\label{Theorem: reduction theory in real rank 1}\cite[Theorem 0.6]{Garland_Raghunathan_1970_Fundamental_domains_for_lattices_MR267041}
There exists $t_0\in \mathbb{R}$ and an open relatively compact subset $\eta_0\subset N$ such that 
\begin{itemize}
    \item[1.] For all $b\in \Xi$, $b^{-1}N b/ b^{-1}N b\cap \Gamma$ is compact;
    
    \item[2.] For all $t>t_0$, and all open, relatively compact subset $\eta$ of $N$ such that $\eta_0\subset \eta$,
   \[G=\bigcup_{b\in \Xi} KA_{t}\eta b \Gamma;\]
   
   \item[3.] Given $t\geq t_0$, $\eta \supset \eta_0$, we can find $t'\in\mathbb{R}$ so that $t'<0$, and for all $\gamma\in \Gamma$, $b,b'\in \Xi$, such that $KA_{t'}\eta b\gamma\cap KA_{t}\eta b'\neq \emptyset$, we must have $b=b'$ and $b\gamma b^{-1}\in {^{\circ}\!P}$.
\end{itemize}

\end{theorem}

By \cite[Theorem 2.1]{Raghunathan_1972_Discrete_subgroups_of_Lie_groups_MR0507234} and \cite[Lemma 3.1]{Dani_Margulis_1993_Limit_distributions_of_orbits_of_unipotent_flows_and_values_of_quadratic_forms_MR1237827}, $\rho(\Gamma)p_N$ is discrete. Here $\rho=Ad\wedge\cdots\wedge Ad$ is the wedge product of adjoint representation of $G$ on $\bigwedge^d \mathfrak{g}$ with $d=\dim N$, and $p_N$ is the representative of $\exp^{-1}(N\cap \Gamma)$ in $\bigwedge^d \mathfrak{g}$. As before $\norm{\cdot}$ is a $\rho(K)$-invariant norm on $\bigwedge^d \mathfrak{g}$, and $\pi:G\to G/\Gamma$ is the natural projection. We have the following compactness criterion for subsets of non-arithmetic quotient $G/\Gamma$.

\begin{lemma}\label{Lemma: compactness criterion in real rank 1}
Let $L\subset G$ be a subset. Then $\pi(L)$ is precompact in $G/\Gamma$ if and only if there exist $\epsilon>0$ such that for all $b\in \Xi$, $g\in L$, one has
\[\inf_{\gamma\in \Gamma}\norm{\rho(g\gamma b^{-1})p_N}>\epsilon.\]
\end{lemma}
\begin{proof}
Assume that $\pi(L)$ is not precompact, then there exists a sequence $\{g_n\}_{n\in \mathbb{N}}\subset L$ such that $\pi(g_n)\to \infty$ as $n\to \infty$. By Theorem \ref{Theorem: reduction theory in real rank 1}, there exist $t\in \mathbb{R}$ and a compact subset $C\subset G$ such that $G=CA_t\Xi\Gamma$. Therefore, we can write $\pi(g_n)=\pi(c_n a_n b_n)$, where $c_n\in C,a_n\in A_t, b_n\in \Xi$, and $\alpha(a_n)\to -\infty$. As $\Xi$ is a finite set, by passing to a subsequence, we may assume that $b_n=b$ for some $b\in \Xi$ and all $n\in \mathbb{N}$. So $g_n=c_n a_n b \gamma_n$. Since $\rho(b\Gamma b^{-1})p_N$ is discrete, so is $\rho(\Gamma b^{-1})p_N$, we have 
\[\inf_{\gamma\in \Gamma}\norm{\rho(\gamma b^{-1})p_N}>0.\]
Note that 
\[\norm{\rho(g_n\gamma_n^{-1}b^{-1})p_N}=\norm{\rho(c_n a_n)p_N}\xrightarrow{n\to \infty}0.\]
Therefore $\inf_{\gamma\in\Gamma}\norm{\rho(g_n\gamma b^{-1})p_N}\to 0$ as $n\to \infty$.

Conversely, if $\pi(L)$ is precompact, there exists $\delta\in \mathbb{R}$ such that for all $g\in L$, if we write $g=c_g a_g b_g \gamma_g$, then $\alpha(a_g)>\delta$. As for all $b\in \Xi$, $\rho(\Gamma b^{-1})p_N$ is discrete, there exist $\epsilon>0$ such that 
\[\inf_{\gamma\in \Gamma}\norm{\rho(g\gamma b^{-1})p_N}>\epsilon,\quad \forall g\in L, b\in \Xi.\]
\end{proof}

We shall need the following lemma for "separation of cusps" using the representation of $G$ on $\bigwedge^d \mathfrak{g}$ (cf. \cite[Lemma 3.2]{Lindenstrauss_Mohammadi_2022_Polynomial_effective_density}). This lemma is analogous to Proposition \ref{Proposition: existence of Zassenhauss neighborhood}.

\begin{lemma}\label{Lemma: separation of cusps}
Let $t_0\in \mathbb{R}$ and $\eta_0\subset N$ be given as in Theorem \ref{Theorem: reduction theory in real rank 1}. Then for any $t\geq t_0$ and $\eta \supset \eta_0$, there exists $t'\in \mathbb{R}$ such that the following holds: For any $b\in \Xi$, $\gamma\in \Gamma$ and $g\in KA_t\eta b\gamma^{-1}$, if $b'\in \Xi$ and $\gamma'\in \Gamma$ are such that
\[\norm{\rho(g\gamma' b'^{-1})p_N}\leq e^{t'},\]
then $b'=b$ and $\gamma'\in \gamma b^{-1}{^{\circ}\!P}b.$
\end{lemma}
\begin{proof}
Let $t'$ be given as in Theorem \ref{Theorem: reduction theory in real rank 1}. Using $G=KANb'$, we may write $g\gamma'=k_1 a_1 n_1 b'$. Then 
\[\norm{\rho(g\gamma' b'^{-1})p_N}=\norm{\rho(k_1 a_1)p_N}=\exp(\alpha(a_1))\leq e^{t'}.\]
Note that as $N/N\cap b'\Gamma b'^{-1}$ is compact, we can find $\gamma_1\in \Gamma$ such that $b'\gamma_1 b'^{-1}\in N$ and $n_1 b' \gamma_1 b'^{-1}\in \eta$. Therefore, $g \in KA_t \eta b' \gamma_1^{-1}\gamma'^{-1}$.

Thus by assumption, we have
\[KA_t\eta b \gamma^{-1} \cap KA_t \eta b' \gamma_1^{-1}\gamma'^{-1}\neq \emptyset.\]
Then Theorem \ref{Theorem: reduction theory in real rank 1} (3) yields $b=b'$, and $\gamma'\in \gamma b^{-1} {^{\circ}\!P}b.$ 
\end{proof}
We define an equivalence relation in the product $\Xi\times\Gamma$ as follows: $(b,\gamma)\sim(b',\gamma')$ if and only if $b=b'$ and $\gamma'\in \gamma b^{-1}{^{\circ}\!P}b$. One can directly verify that this is indeed an equivalence relation. The following is an immediate consequence of Lemma \ref{Lemma: separation of cusps}.
\begin{corollary}\label{Corollary: separation of cusps}
Let $t_0\in \mathbb{R}$ and $\eta_0\subset N$ be given as in Theorem \ref{Theorem: reduction theory in real rank 1}. Then there exists $t'\in \mathbb{R}$ such that the following holds: For any $g\in G$, $(b,\gamma)\in \Xi\times \Gamma$, if
\[\norm{\rho(g\gamma b^{-1})p_N}\leq e^{t'},\]
then for any $(b',\gamma') \in \Xi\times \Gamma$ such that $(b',\gamma')\nsim (b,\gamma)$,
\[\norm{\rho(g\gamma' b'^{-1})p_N}> e^{t'}.\]
\end{corollary}

Now let $\omega:G\to G/P$ be the natural projection map. 
\begin{lemma}\label{Lemma: pushout in real rank 1}
There exists $C>1$ such that for any $g\in G$, there exists $a\in A$ such that
\[\norm{\rho(ag)p_N}>C\norm{\rho(g)p_N}\]
\end{lemma}
\begin{proof}
Using Bruhat decomposition, we can write $g=uwza_0v$, where $u,v\in N$, $w\in W(G),a_0\in A,z\in M$. Note that for any $a\in A$, we have $\rho(a)p_N=\exp(\chi_0(a)) p_N$, where $\chi_0=m\alpha$ for some $m\in \mathbb{N}$. Therefore, 
\[\rho(g)p_N=\rho(uwa_0)p_N\in V_{ w(\chi_0)}\oplus \bigoplus_{\chi> w(\chi_0)} V_{\chi}.\]
Since $\boldsymbol{A}$ is a maximal $\mathbb{R}$-split torus, we can choose $a\in A$ such that $w(\chi_0)\geq \chi$ for $\chi>w(\chi_0)$, and $w(\chi_0)(a)>0$. Therefore, there exists $a\in A$ such that  
\[\norm{\rho(ag)p_N}>C\norm{\rho(g)p_N}.\]
\end{proof}

\begin{proposition}\label{Proposition: maximal R split torus is UN in real rank 1}
Let $\boldsymbol{A}$ be a maximal $\mathbb{R}$-split torus of $\boldsymbol{G}$, then the action of $A$ on $G/\Gamma$ is uniformly non-divergent.
\end{proposition}
\begin{proof}
By compactness of $G/P$ and a standard continuity argument, we can find a finite set $\{a_1,\cdots, a_n\}\subset A$ (indeed $n=2$) and $C>1$ such that for any $g\in G$, there exists $a_i$ for some $1\leq i\leq n$ such that
\[\norm{\rho(a_i g)p_N}>C\norm{\rho(g)p_N}.\]
Let $c=\max_{1\leq i\leq n}\{\norm{\rho(a_i)},\norm{\rho(a_i^{-1})}\}\geq 1$, where $\norm{\rho(a_i)}$ denotes the operator norm of $\rho(a_i)$. Let $t'$ be given as in Corollary \ref{Corollary: separation of cusps}. Assume that $g\in G$ is such that there exists $(b,\gamma)\in \Xi\times \Gamma$ with 
\[\norm{\rho(g\gamma b^{-1})}< c^{-1}e^{t'},\]
then by Corollary \ref{Corollary: separation of cusps}, for any $(b',\gamma')\nsim (b,\gamma)$,
\[\norm{\rho(g\gamma' b'^{-1})p_N}>e^{t'}.\]
By Lemma \ref{Lemma: pushout in real rank 1}, there exists $a_i$ such that \[\norm{\rho(a_ig\gamma b^{-1})p_N}>C \norm{\rho(g\gamma b^{-1})p_N},\]
while by the choice of $c$, for any other $(b',\gamma')\nsim (b,\gamma)$,
\[\norm{\rho(a_i g\gamma' b'^{-1})p_N}>c^{-1}e^{t'}.\]
Therefore, by Lemma \ref{Lemma: compactness criterion in real rank 1}, the proposition follows.
\end{proof}

\begin{proof}[Proof of Theorem \ref{Theorem: nonarithmetic case}]
If $\boldsymbol{H}$ contains a maximal $\mathbb{R}$-split torus, then by Proposition \ref{Proposition: maximal R split torus is UN in real rank 1}, the action of $H$ on $G/\Gamma$ is uniformly non-divergent.

Conversely, assume that $\boldsymbol{H}$ does not contain a maximal $\mathbb{R}$-split torus. We may write $\boldsymbol{H}=\boldsymbol{SV}$, where $\boldsymbol{S}$ is reductive and $\boldsymbol{V}$ is the unipotent radical of $\boldsymbol{H}$. Since $\boldsymbol{G}$ is of rank 1, $S$ is compact. Also, by conjugating a suitable element in $G$, we may assume that $\boldsymbol{V}\subset \boldsymbol{N}$. Let $a_t\in A$ be a sequence such that $\alpha(a_t)\to -\infty$ as $t\to \infty$, we have \[\sup_{h\in H}\norm{\rho(ha_t)p_N}\to 0, \quad \text{ as }t\to \infty.\]
Therefore, by Lemma \ref{Lemma: compactness criterion in real rank 1}, the action of $H$ on $G/\Gamma $ is not uniformly non-divergent.
\end{proof}

\begin{ack}
We thank Professor Yitwah Cheung and Professor Lingming Liao for their interest in this project, and their encouragement. We also thank the anonymous referee for helpful comments on an earlier version of this paper. H.Z. is supported by the startup grant of Soochow University and the National Natural Science Foundation of China (No. 12501250). R.Z. is supported by National Natural Science Foundation of China (No. 12201013).
\end{ack}

\bibliographystyle{plain}
\bibliography{references.bib}

\end{document}